\documentclass{amsart}
\usepackage{latexsym}

\usepackage{amsmath,amsthm,amsfonts,amscd,eucal}
\usepackage{graphicx}
\usepackage{epsfig}
\usepackage{epstopdf} 

\numberwithin{equation}{section}

\def\cb{{\mathcal B}}

\def\cd{{\mathcal D}}

\def\ch{{\mathcal H}}

\def\car{{\mathcal R}}

\def\bc{{\mathbb C}}
\def\bg{{\mathbb G}}
\def\bh{{\mathbb H}}

\def\bn{{\mathbb N}}

\def\br{{\mathbb R}}
\def\bt{{\mathbb T}}

\def\bz{{\mathbb Z}}

\def\a{\alpha}
\def\b{\beta}
\def\g{\gamma}        \def\G{\Gamma}
\def\d{\delta}        \def\D{\Delta}
\def\eps{\varepsilon}
    
\def\z{\zeta}
\def\th{\vartheta}

\def\l{\lambda}       \def\La{\Lambda}
\def\m{\mu}
\def\n{\nu}

\def\r{\rho}
\def\s{\sigma}        \def\S{\Sigma}
    
\def\t{\tau}

\def\f{\varphi} \def\F{\Phi}

\newtheorem{Thm}{Theorem}[section]
\newtheorem{Cor}[Thm]{Corollary}
\newtheorem{Prop}[Thm]{Proposition}
\newtheorem{Lemma}[Thm]{Lemma}

\theoremstyle{definition}
\newtheorem{Dfn}[Thm]{Definition}

\theoremstyle{remark}

\def\di{\mathop{\rm d}\!}
\def\rel{\mathop{\rm Re}}
\def\im{\mathop{\rm Im}}
\def\spr{\mathop{\rm spr}}

\def\tr{\mathop{\rm Tr}{}\!}

\def\idd{{\bf 1}\!\!{\rm I}}

 \def\di{\mathop{\rm d}\!}

\newcommand{\nn}{\nonumber}

\begin{document}

\title[harmonic analysis on cayley trees]
{harmonic analysis on perturbed Cayley Trees}
\author{Francesco Fidaleo}
\address{Dipartimento di Matematica,
Universit\`{a} di Roma Tor Vergata,
Via della Ricerca Scientifica 1, Roma 00133, Italy}

\email{fidaleo@mat.uniroma2.it}


\date{\today}

\begin{abstract}
We study some spectral properties of  the adjacency operator of non
homogeneous networks. The graphs under investigation are obtained by
adding density zero perturbations to the homogeneous Cayley Trees.
Apart from the natural mathematical meaning, such spectral
properties are relevant for the Bose Einstein Condensation for the
pure hopping model describing arrays of Josephson junctions on non
homogeneous networks. The resulting topological model is described
by a one particle Hamiltonian which is, up to an additive constant, the
opposite of the adjacency operator on the graph. It is known that
the Bose Einstein condensation already occurs for unperturbed
homogeneous Cayley Trees. However, the particles condensate on the
perturbed graph, even in the configuration space due to
nonhomogeneity. Even if the graphs under consideration are
exponentially growing, we show that it is enough to perturb in a
negligible way the original graph in order to obtain a new network
whose mathematical and physical properties
dramatically change. Among the results proved in the present paper,
we mention the following ones. The appearance of the {\it Hidden
Spectrum} near the zero of the Hamiltonian, or equivalently below
the norm of the adjacency. The latter is related with the value of
the critical density and then with the appearance of the
condensation phenomena. The investigation of the {\it recurrence/transience
character} of the adjacency, which is connected to the possibility
to construct locally normal states exhibiting the Bose Einstein
condensation. Finally, the study of the {\it volume growth of the
wave function} of the ground state of the Hamiltonian, which is
nothing but the generalized Perron Frobenius eigenvector of the
adjacency. This Perron Frobenius weight describes the spatial
distribution of the condensate and its shape is connected with the
possibility to construct locally normal states exhibiting the Bose
Einstein condensation at a fixed density greater than the critical
one. 
\end{abstract}

\maketitle

\centerline{DEDICATO A BERTA}

\section{introduction}

The present paper is devoted to the analysis of the mathematical properties of non homogeneous
networks obtained by adding density zero perturbations to homogeneous Cayley Trees, the latter being the Cayley graphs of free (products of) groups, see e.g. Fig. \ref{Figc} and Fig. \ref{Fig8}. As explained in the previous paper \cite{FGI1}, such mathematical properties are deeply connected with the Bose Einstein condensation (BEC for short)
of Bardeen Cooper pairs in networks describing arrays of Josephson junctions (see e.g. Section 62 of \cite{LL}, and \cite{BC}). The formal Hamiltonian describing such arrays of Josephson junctions is the quartic Bose Hubbard Hamiltonian, given on a generic network $G$ by
\begin{equation}
\label{boha}
H_{BH}=m\sum_{i\in VG}n_i+\sum_{i,j\in VG}A_{ij}\big(Vn_in_j-J_0a^{\dagger}_ia_j\big)\,.
\end{equation}
Here, $VG$ denotes the set of the vertices of the network $G$, $a^{\dagger}_i$ is the Bosonic creator, and $n_i=a^{\dagger}_ia_i$ the number operator
on the site $i\in VG$ (cf. \cite{BR2}). Finally, $A$ is the adjacency operator whose matrix element
$A_{ij}$ in the place $ij$ is the number of the edges connecting the site $i$ with the site $j$ (in particular it is Hermitian). It was argued in \cite{BCRSV} that, in the case when $m$ and $V$ are negligible with respect to $J_0$, the hopping term dominates the physics of the system. Thus, under this approximation, \eqref{boha} becomes the {\it pure hopping Hamiltonian} given by
 \begin{equation}
\label{boha1}
 H_{PH}=-J\sum_{i,j\in VG}A_{ij}a^{\dagger}_ia_j\,,
\end{equation}
 where the constant $J>0$ is a mean field coupling constant which might be different from the $J_0$ appearing in the more realistic Hamiltonian \eqref{boha}.\footnote{It is of course a very interesting problem to provide a theoretical estimate of the coupling constant $J$ appearing in the pure hopping Hamiltonian. However, it might be reasonable to accept the idea that, at very low temperature when the thermal agitation plays a negligible role, the pure hopping term dominates the remaining ones in \eqref{boha}.} Recently, in some crucial experiments (cf. \cite{SRC}), it was found an enhanced current at low temperatures for non homogeneous arrays of Josephson junctions, which might be explained via the Bose Einstein condensation. On the other hand, it was showed in Theorem 7.6 of \cite{FGI1}, that for free models (i.e.
when $V=0$ in \eqref{boha}), the condensation phenomena can occur after adding a negligible number of edges, only if the Hamiltonian is pure hopping.

It is well known (cf. \cite{BR2}, Section 5.2) that most of the physical properties of the quadratic multi particle Hamiltonian \eqref{boha1} are encoded into the spectral properties of the one particle Hamiltonian
\begin{equation}
\label{boha3}
 h=-JA\,,
 \end{equation}
naturally acting on $\ell^2(VG)$.

In light of the previous considerations, it is natural to address the investigation of the
pure hopping mathematical model described by the Hamiltonian obtained by putting $J=1$
 in \eqref{boha3}, and normalizing to ensure the positivity of the energy. The resulting one particle Hamiltonian for the purely topological model under consideration is then
 \begin{equation}
\label{boha2}
 H=\|A\|\idd-A\,,
 \end{equation}
 where $A$ is the adjacency of the fixed graph $G$, acting on the Hilbert space
 $\ell^2(VG)$.

One of the first mathematical attempts to investigate the BEC on non homogeneous amenable graphs, such as the Comb graphs, was made in \cite{BCRSV}. In that paper, it was pointed out that there appears an {\it hidden spectrum}, which is responsible for the finiteness of the critical density. In addition, the behavior of the wave function of the ground state, describing the spatial density of the condensate, was also computed. Some spectral properties of the Comb and the Star graph (cf. Fig. \ref{Figa}) were investigated in
\cite{ABO} in connection with the various  notions of independence in Quantum Probability. In that paper, it was noticed the possible connection between such spectral properties and the BEC.
 \begin{figure}[ht]
     \centering
     \psfig{file=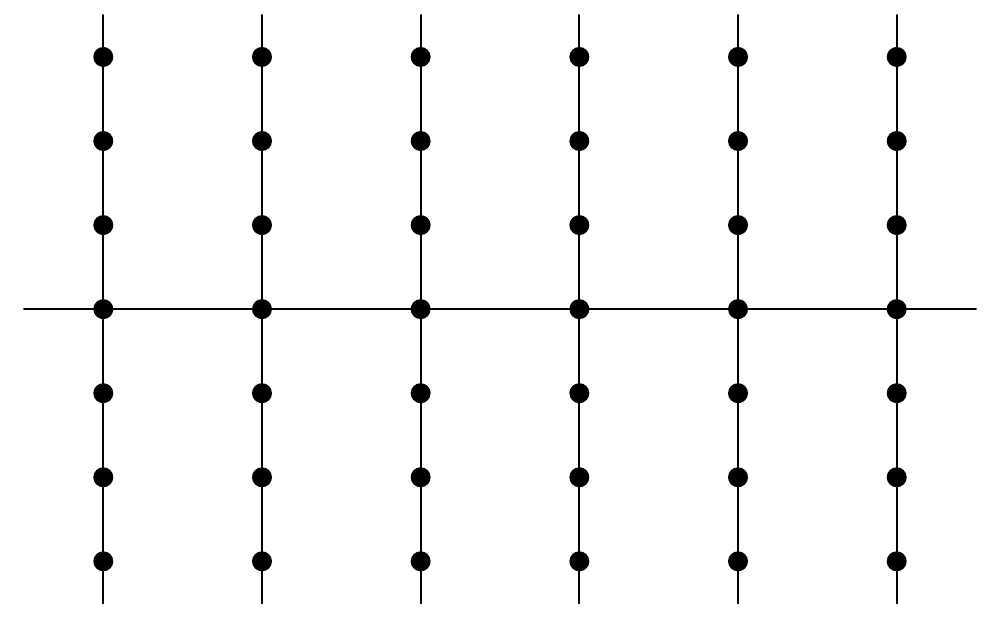,height=1.2in} \qquad \qquad
     \psfig{file=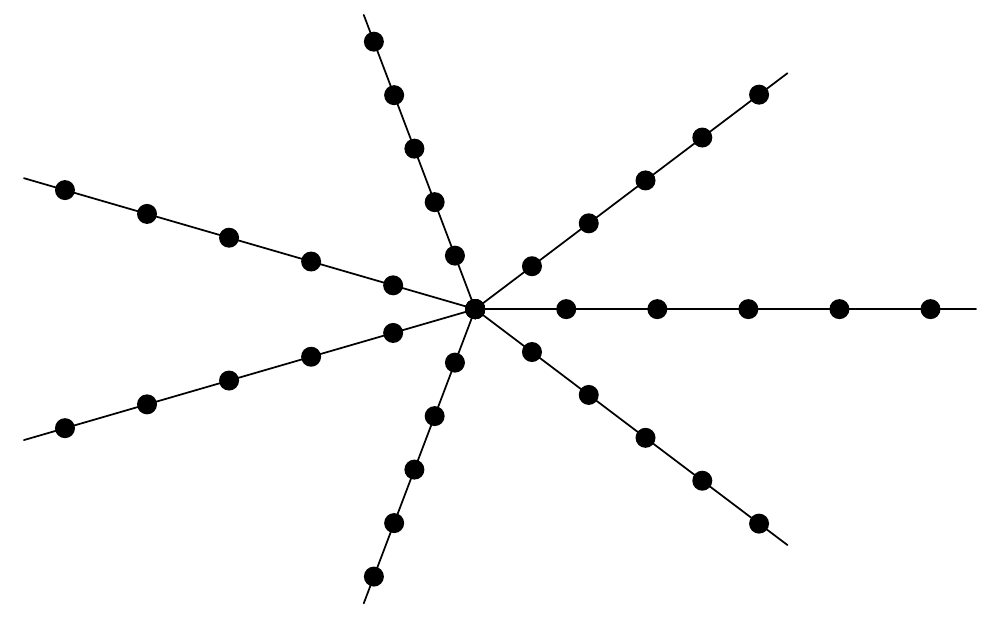,height=1.2in}
     \caption{Comb and Star Graphs.}
     \label{Figa}
     \end{figure}

The systematic investigation of the BEC for the pure hopping model
on a wide class of amenable networks obtained by negligible perturbations of periodic graphs, has been started in \cite{FGI1}. The emerging results are quite surprising. First of all, the appearance of the hidden spectrum was proven for most of the graphs under consideration. This is due to the combination of two opposite phenomena arising from the perturbation. If the perturbation is sufficiently large (in many cases it is enough a finite perturbation), the norm
$\|A_p\|$ of the adjacency of the perturbed graph becomes larger than the analogous one $\|A\|$ of the unperturbed adjacency. On the other hand, as the perturbation is sufficiently small (i.e. zero--density), the part of the spectrum
$\s(A_p)$ in the segment $(\|A\|, \|A_p\|]$ does not contribute to the density of the
states.\footnote{Due to the standard normalization chosen in the present paper, the integrated density of the states describing the density of the eigenvalues, is a cumulative function $F$ whose support is included in the closed line
$\overline{\br_+}$. See Section \ref{giem} below, and the reference cited therein.}
This allows us to compute the critical density $\r_c(\b)$ at the inverse temperature $\b$ for the perturbed model by using the integrated density of the states $F$ of the unperturbed one,
 \begin{equation}
\label{cdens}
\r_c(\b)=\int\frac{\di F(x)}{e^{\b\left(x+(\|A_p\|-\|A\|)\right)}-1}\,.
\end{equation}
The resulting effect of the perturbed model exhibiting the hidden spectrum (i.e. when $\|A_p\|-\|A\|>0$) is that the critical density is always finite.\footnote{Compare with the Lifschitz tails in randomly perturbed Hamiltonians, see e.g. \cite{JPZ, LZ}.}

Another relevant fact connected with the introduction of the perturbation, and thus to the non homogeneity, is the possible change of the transience/recurrence character (cf. \cite{S}, Section 6) of the adjacency operator. It has to do with the possibility to construct locally normal states exhibiting BEC.\footnote{For the possible applications to Probability Theory of the
transience character of an infinite matrix with non negative entries, the reader is referred to
\cite{S}.} As explained in \cite{FGI1}, the last relevant fact is the investigation of the shape of wave function of the ground state of the model, describing the spatial distribution of the condensate on the network in the ground state of the Hamiltonian. From the mathematical viewpoint, this is nothing but the Perron Frobenius generalized eigenvector of the adjacency (cf. \cite{PiWo, S}).

It appears clear that the physical and the mathematical aspects of the topological model based on the pure hopping Hamiltonian \eqref{boha2} are strongly related. This can be understood also in the following simple way. For Bosonic models, described by the Canonical Commutation Relations (cf. \cite{BR2}), most of the physical relevant quantities are computed by the functional calculus of suitable functions of the one particle Hamiltonian. The critical density  \eqref{cdens} is one of them. But, the asymptotic behavior of the Hamiltonian \eqref{boha2} near zero corresponds to the asymptotics of the spectrum of $A$ close to $\|A\|$. Indeed, by the Taylor expansion, we heuristically get for the function appearing in the Bose Gibbs occupation number (cf. \cite{LL}, Section 54) at small energies, for the chemical potential $\m<0$,
$$
\frac1{e^{H-\m\idd}-1}\approx (H-\m\idd)^{-1}=((\|A\|-\m)\idd-A)^{-1}
\equiv R_A(\|A\|-\m)\,.
$$
Then the study of the BEC is reduced to the investigation of the spectral properties of the resolvent
$R_A(\l)$, for $\l\approx\|A\|$.
\begin{figure}[ht]
     \centering
     \psfig{file=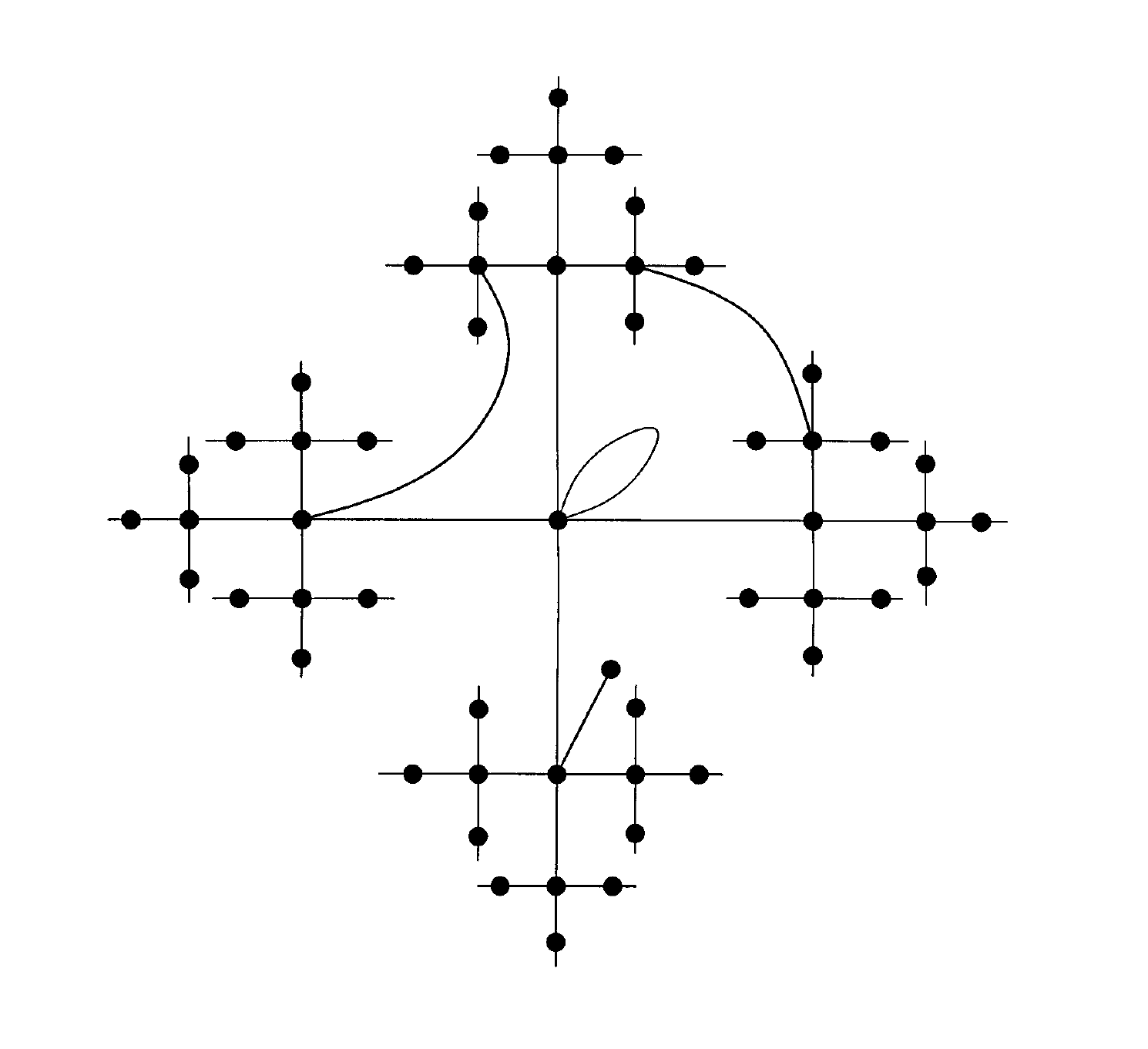,height=2.5in}
     \caption{Finite additive perturbation of the Cayley Tree of degree 4.}
     \label{Figb}
     \end{figure}

The networks under consideration in the present paper are density zero additive perturbations of
exponentially growing graphs made of homogeneous Cayley Trees, see Fig. \ref{Figb}. We restrict our analysis to the mathematical aspects explained below.
Among the models treated in the present paper, we mention the perturbations $\bg^{Q,q}$,
$2\leq q<Q$, and $\bh^Q$, of the homogeneous Cayley Tree $\bg^{Q}$ along a subtree isomorphic to $\bg^{q}$, and $\bn$ respectively, see below. For these situations, we are able to write down and solve the secular equation. Thus, we can determine the $q$, $Q$ for which
$\bg^{Q,q}$ admits the hidden spectrum. In addition, we provide a useful formula for the resolvent of $A_{\bg^{Q,q}}$. Thus, we can write down the Perron Frobenious eigenvector obtained as the infinite volume limit of the finite volume Perron Frobenius eigenvectors (normalized to $1$ at a fixed root), and finally determine whether the perturbed graph is recurrent or transient.

A result which is in accordance with the intuition (i.e. suggested by the shape of the Perron Frobenius vector), and with the previous ones described in \cite{FGI1}, is that the 
recurrence/transience character of
$\bg^{Q,q}$ and $\bh^Q$, is determined by that of the base point of the perturbation. Namely,
$\bg^{Q,2}$ is recurrent as
$\bg^{2}\sim\bz$. The network $\bh^Q$ is transient as the base point of its perturbation, which is isomorphic to $\bn$ (cf. \cite{FGI1}, Proposition 8.2). Finally, if $q>2$,
$\bg^{Q,q}$ is transient as well, being $\bg^{q}$ transient when $q>2$.

As previously explained, all the results listed below have relevant physical applications to the BEC. We postpone the detailed investigation of such applications to the forthcoming paper
\cite{F}.

\section{preliminaries}
\label{giem}

 In the present paper, a {\it graph} (called also a {\it network}) $X=(VX,EX)$ is a collection $VX$ of objects,
 called {\it vertices}, and a collection $EX$ of unordered lines connecting vertices, 
 called {\it edges}. Denote $E_{xy}$ the collection of all the edges connecting $x$ with $y$. As the
 edges are unordered, $E_{xy}=E_{yx}$. Two vertices $x$, $y$ are said to be
 {\it adjacent} if there exists an edge $e_{xy}\in E_{xy}$ joining $x$, $y$. In this situation, we write 
 $x\sim y$.
 
Let us denote by $A=[A_{xy}]_{x,y\in X}$, $x,y\in VX$, the {\it adjacency
 matrix} of $X$, that is,
 $$
 A_{xy}=|E_{xy}|\,.
 $$ 
Notice that all the geometric properties of $X$ can be
 expressed in terms of $A$.  For example, a graph is connected, that is
 any two different vertices are joined by a path, if and only if  $A$ is irreducible. In addition,
 the {\it degree} $\deg(x)$ of a
 vertex $x$, that is the number of the incoming (or equivalently outcoming) edges of $x$ is
 $\langle A^*A\d_x,\d_x\rangle$. Setting
 $$
 \deg:=\sup_{x\in VX} \deg(x)\,,$$
 we have
 $\sqrt{\deg}\leq\|A\|\leq \deg$, that is $A$ is bounded if and only if $X$
 has uniformly bounded degree.  We denote by $D=[D_{xy}]_{x,y\in X}$ the {\it degree matrix} of
 $X$, that is,
$$
D_{xy}:=\deg(x)\d_{x,y}.
$$
The {\it Laplacian} on the graph is $\D=A-D$. The definition used here implies $\D<0$, and is the standard one adopted in the physical literature.

In the present paper, all the graphs are connected, countable and with uniformly bounded
degree. In addition, we deal only with bounded operators acting on $\ell^2(VX)$ if it is not otherwise specified.

Let $B$ be a closed operator acting on
$\ell^2(VX)$, and $\l\in{\rm P(B)}\subset\bc$ the resolvent set of $B$. As usual,
$$
R_B(\l):=(\l\idd-B)^{-1}
$$
denotes the {\it resolvent} of $B$.

Fix a bounded matrix with positive entries $B$ acting on $\ell^2(VX)$. Such an operator is called {\it positive preserving} as it preserves the elements of $\ell^2(VX)$ with positive entries.
A sequence $\{v(x)\}_{x\in VX}$ is called a {\it (generalized) Perron Frobenius eigenvector} if it has positive entries and
 $$
 \sum_{y\in VX}B_{xy}v(y)=\|B\|v(x)\,,\quad x\in VX\,.
 $$
Suppose for simplicity that $B$ is selfadjoint. It is said to be {\it recurrent} if
\begin{equation}
\label{caz}
\lim_{\l\downarrow\|B\|}\langle R_B(\l)\d_x,\d_x\rangle=+\infty\,.
\end{equation}
otherwise $B$ is said to be {\it transient}.
It is shown in \cite{S}, Section 6, that the recurrence/transience character of $B$ does not depend on the base point chosen for computing the limit in \eqref{caz}.

The Perron Frobenius eigenvector is unique up to a multiplicative constant, if $X$ is finite or when $B$ is recurrent, see e.g. \cite{S}. It is unique also for the adjacency on the tree like networks (cf. \cite{PiWo}). In general, it is not unique, see e.g. \cite{FGI2} for the cases relative to the Comb graphs.
\begin{figure}[ht]
     \centering
     \psfig{file=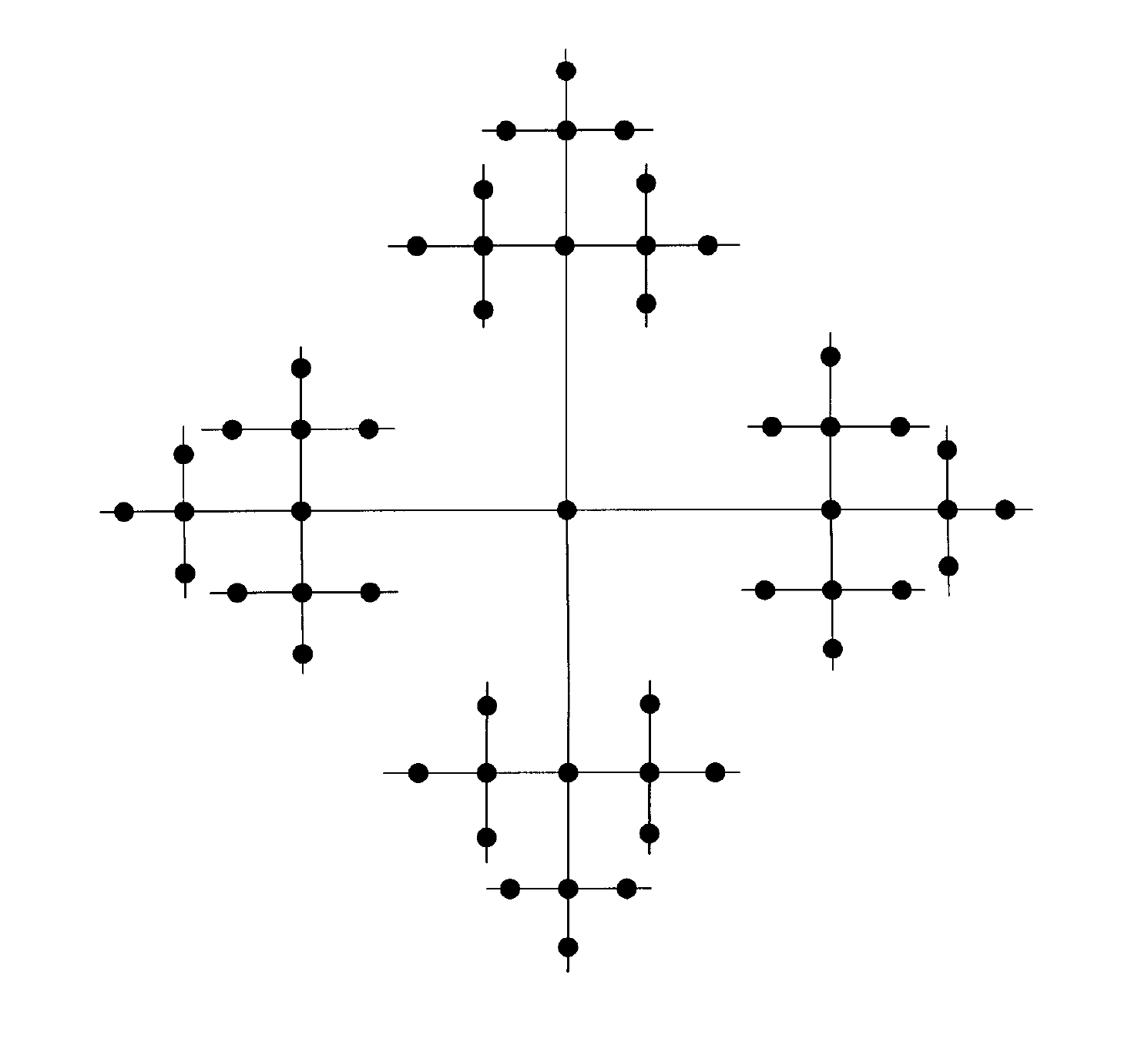,height=2.5in}
     \caption{The Cayley Tree of degree 4.}
     \label{Figc}
     \end{figure}

We say that an operator $B$ acting on $\ell^2(VX)$ has {\it finite
 propagation} if there exists a constant $r=r(A)>0$ such that, for any
 $x\in X$, the support of $Av$ is contained in the (closed) ball
 $B(x,r)$ centered in $x$ and with radius $r$. It is easy to show that if $A$ is the adjacency operator on $X$, then $A^k$ has propagation $k$ for any integer $k\geq0$.

The graphs we deal with in our analysis are (additive, negligible) perturbations of homogeneous Cayley Trees if it is not otherwise specified. The reader is referred to \cite{Z} for the definitions and the main properties concerning the Cayley Trees.

Let $X$ be any Cayley Tree of degree $q$, see Fig. \ref{Figc}. Fix a root $0\in X$ and consider the ball $X_n$ including all the vertices at distance less than or equal to $n$ from
$0$, see Fig \ref{Fig5}. We denote by $d$ the canonical distance on $X$, where $d(x,y)$ is the number of the edges of the minimal path connecting $x$ with $y$.
\begin{figure}[ht]
     \centering
     \psfig{file=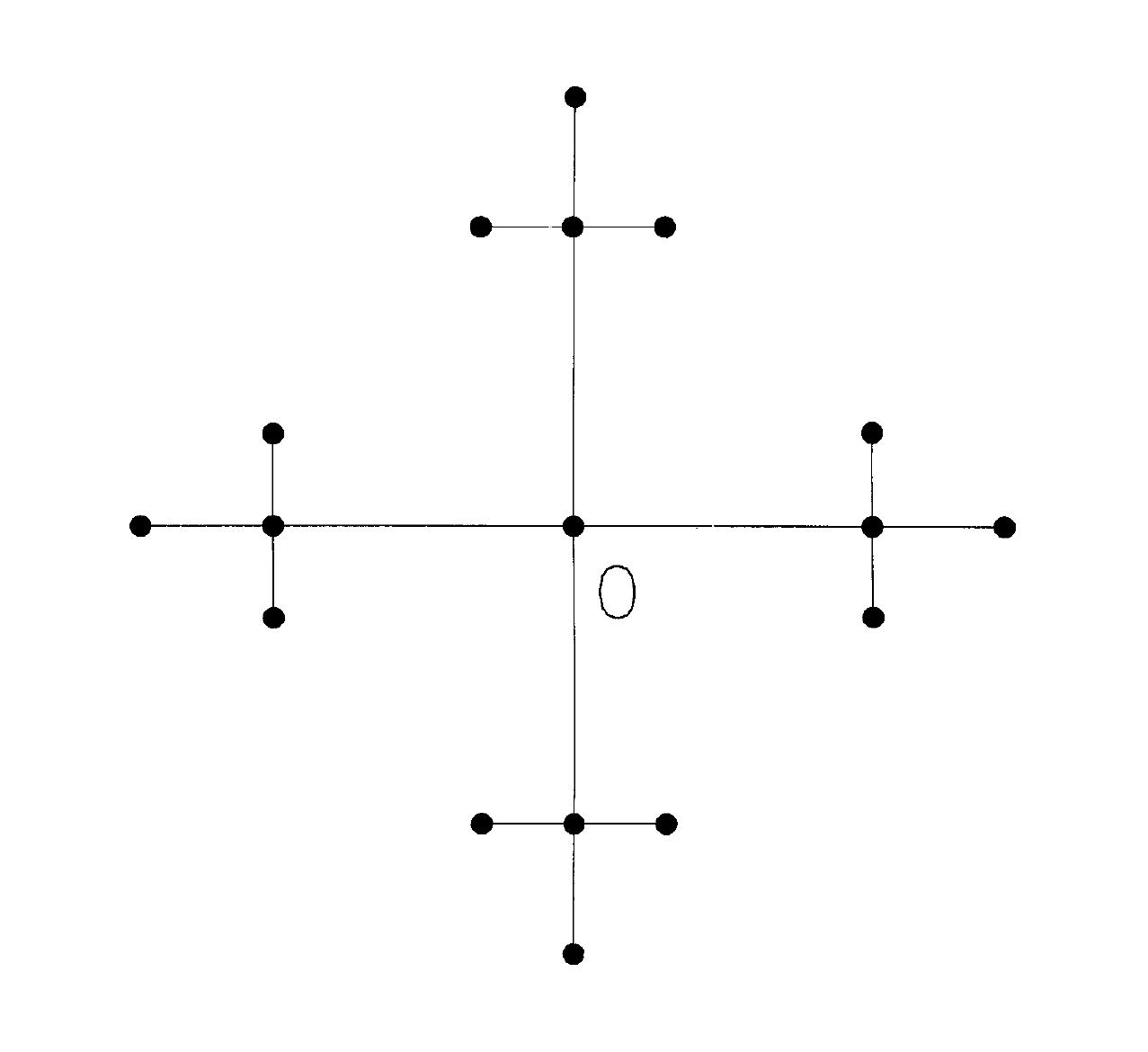,height=2in}
     \caption{The ball of radius 2 in the Cayley Tree of degree 4.}
     \label{Fig5}
     \end{figure}
Let $A_{X_n}$, $A_X$ be the adjacency matrices of the corresponding graphs. The formers are nothing but the restriction of the latter to the graphs $X_n$:
$$
A_{X_n}=P_nA_XP_n\lceil_{\ell^2(VX_n)}
$$
where $P_n$ is the orthogonal projection onto $\ell^2(VX_n)$.

One of the most useful objects for infinite systems like those considered in the present paper is the so called integrated density of the states. We start with the following definition.
Consider on $\cb(\ell^2(VX))$ the state
$$
\t_n:=\frac1{|VX_n|}\tr_n(P_n\,{\bf\cdot}\,P_n)\,,
$$
$P_n$ being the selfadjoint projection onto $\ell^2(VX_n)$. Define for a bounded operator $B$,
\begin{equation}
\label{3434}
\t(B):=\lim_n\t_n(B)\,,\quad B\in\cd_\t\,,
\end{equation}
where the domain $\cd_\t$ is precisely the linear submanifold of $\cb(\ell^2(VX))$ for which
the limit in \eqref{3434} exists.
Let $B\in\cb(\ell^2(X))$ be a bounded selfadjoint operator. We suppose for simplicity that $B$ is positive and $\min\s(B)=0$.
Suppose in addition that $\{f(B)\mid f\in C(\br)\}\subset\cd_\t$.
Then $\m_B(f):=\t(f(B))$ defines a positive normalized functional on $C(\s(B))$, and then a probability measure on the (positive) real line by the Riesz Markov Theorem. Thus, there exists a unique increasing right continuous function $x\in\br\mapsto N_B(x)\in\br$ satisfying
$$
N_B(x)=0\,,\,\,x<0\,,\quad N_B(x)=1\,,\,\,x\geq\|B\|\,,
$$
such that
$$
\m_B(f)=\int f(x)\di N_B(x)\,,
$$
where the last integral is  a Lebesgue Stieltjes integral, see
\cite{Ro}, Section 12.3. Such a cumulative function $N_B$ is called
the {\it integrated density of the states} of $B$, see e.g.
\cite{PF}.

Let $B\in\cb(\ell^2(VX))$ be a selfajoint operator with $\min\s(B)=0$, such that
$\{f(B)\mid f\in C(\br)\}\subset\cd_\t$. Let
$N_B$ its integrated density of the states.
\begin{Dfn}
We say that $B$ exhibits {\it hidden spectrum} if there exist $x_0>0$ such that
$N_B(x)=0$ for each $x<x_0$.
\end{Dfn}
Notice that, if $B$ exhibits hidden spectrum, then the part of the spectrum
$$
\emptyset\neq\s(B)\bigcap[0,x_0)\,,
$$
does not contribute to the integrated density of the states.

Consider the integrated density of the states
$F:=N_{\|A_X\|\idd-A_X}$ of $\|A_X\|\idd-A_X$. This cumulative
function exists, and is the pointwise limit of the densities of the
eigenvectors of the finite volume operators
$\|A_X\|\idd-A_{X_n}$,that is the finite volume density of the
states (up to an additive constant going to zero as $n\to+\infty$),
except on at most a countable set. Indeed, for the inverse
temperature $\b>0$, let
\begin{equation}
\label{4}
\F_n(\b):=\frac1{|VX_n|}\tr_n(e^{-\b(\|A_X\|\idd-A_{X_n})})
\end{equation}
be the one particle finite volume partition function.\footnote{In
the physical language, ($\Phi_n$) $\Phi$ is called the (finite
volume) Gibbs partition function of the model at the inverse
temperature $\b$.} It is nothing but the Laplace transform of the
density of the states of $\|A_X\|\idd-A_{X_n}$. As shown in
\cite{BDP}, it converges pointwise to
\begin{equation}
\label{3}
\F(\b):=\frac{(q-2)^2}{q-1}\sum_{k=1}^{+\infty}\sum_{n=1}^{k}
(q-1)^{-k}e^{-4\b\sqrt{q-1}\sin^2\frac{n\pi}{2(k+1)}}\,,
\end{equation}
as $X_n\uparrow X$.
\begin{Prop}
\label{1}
The $\F(\b)$ in \eqref{3} is the Laplace transform of a cumulative function $F$ of a probability measure on the real line whose support is contained in the interval $[0, 4 \sqrt{q-1}]$.
\end{Prop}
\begin{proof}
The proof directly follows from Theorem XIII.1.2 of \cite{Fe} by
interchanging the summation with the limit $\b\downarrow 0$.
\end{proof}
Now we show that the cumulative function $F$ is nothing but the integrated density of the states of $\|A_X\|-A_X$.
\begin{Prop}
\label{a}
$\{f(A_X)\mid f\in C(\br)\}\subset\cd_\t$ and
\begin{equation*}
 \t(f(A_X))=\int f(\|A_X\|-x)\di F(x)\,,
\end{equation*}
where the Laplace transform of $F$ is the function given in \eqref{3}.
\end{Prop}
\begin{proof}
Let $F_n$ be the inverse Laplace transform of $\F_n$ given in \eqref{4}, see e.g. \cite{Fe}, Chapter XIII. We get
$$
\t_n(f(A_{X_n}))=\int f(\|A_X\|-x)\di F_n(x)\to\int f(\|A_X\|-x)\di F(x)
$$
by Proposition \ref{1} and the first Helly Theorem (cf. \cite{Fe}, Theorem VIII1.1).
\end{proof}
 According to the previous results, the density of the states $F$ is the inverse Laplace transform of $\Phi$,
see e.g. \cite{Fe}, Theorem XIII.4.2.

Consider the graph $Y$ such that $VY=VX$, both equipped with the
same exhaustion $\{VY_n\}_{n\in\bn}$ such that $VY_n=VX_n$,
$n\in\bn$. The graph $Y$ is a {\it negligible} or {\it density zero
perturbation} of $X$ if it differs from $X$ by a number of edges
such that
$$
 \lim_n\frac{|\{e_{xy}\in EX\triangle EY\mid x\in VX_n\}|}{|VX_n|}=0\,,
$$
where $EX\triangle EY$ denotes the symmetric difference. To simplify
matters, we consider only perturbations involving edges, the more
general case involving also vertices can be treated analogously, see
\cite{FGI1}.

The following result, concerning general perturbations obtained by adding and/or
removing edges, of general networks, was proven in \cite{FGI1} (cf.
Theorem 6.1). We report its proof for the convenience of the reader.
Fix $X$ as the reference graph and define $A_X:=A$, $A_Y:=A+D$,
where $D$ is the perturbation, which is considered to eventually act
on $\overline{\car(D)}$. Put, for $\l\in\bc$,
\begin{equation*}
S(\l):=DP_{\overline{\car(D)}}R_A(\l)\lceil_{\overline{\car(D)}}\,.
\end{equation*}
\begin{Lemma}
\label{fgiii}
With the above notation, suppose that $|\car(D)|<+\infty$. Then $\l\in{\rm P}(A)$ is an eigenvalue of $A_Y$ if and only if $1$ is
an eigenvalue of $S(\l)$.  If this is the case, the corresponding
eigenvectors $v$, respectively $w$, are related by
\begin{equation}
\label{zeta}
v=R_{A_X}(\l)w\,,\quad w=DP_{\car(D)}v\,.
\end{equation}
\end{Lemma}
\begin{proof}
Let $\l\not\in\s(A)$, and suppose there is $v\in\ell^2(VX)$ such that $A_{Y}v=\l v$.  From the definition of $D$, we recover $\l v-Av=DP_{\car(D)}v$, which implies
\begin{equation*}
v=R_{A}(\l)DP_{\car(D)}v\,.
\end{equation*}
Multiplying both sides by $DP_{\car(D)}$,
\begin{equation*}
DP_{\car(D)}v=DP_{\car(D)}R_{A}(\l)DP_{\car(D)}v\,.
\end{equation*}
Namely, $w:=DP_{\car(D)}v$ is an eigenvector of $S(\l)$ corresponding to the eigenvalue $1$.
Conversely, let $\l\not\in \s(A)$ and suppose that $w$ is an eigenvector of $S(\l)$ corresponding to the eigenvalue $1$.  Extend $w$ to $0$ outside $\car(D)$. Define
$v:=R_{A}(\l)w$.  Then, it is easy to show that $v$ is an
eigenvector of $A_{Y}\equiv A_X+D$ with eigenvalue $\l$.
 \end{proof}
Let $D_{XY}:=A_X-A_Y$. It is easily seen that
\begin{equation}
\label{estadc}
\tr_n(P_nD_{XY}^2P_n)=|\{e_{xy}\in EX\triangle EY\mid x\in VX_n\}|\,.
\end{equation}
\begin{Prop}
\label{density0} 
Let $Y$ be a negligible perturbation of the tree $X$. Then $\{f(A_Y)\mid f\in C(\br)\}\subset\cd_\t$, and
\begin{equation}
\label{cazcs}
\t(f(A_Y))=\t(f(A_X))\,.
\end{equation}
\end{Prop}
\begin{proof}
We start by noticing that
\begin{equation}
\label{rrefe}
A_Y^k-A_X^k=\sum_{l=1}^{k}A_Y^{k-l}D_{XY}A_X^{l-1}\,.
\end{equation}
As the power $s$ of the adjacency matrix $A$ has propagation $s$, hence
$A_Xv\in B_{n+s}$, provided $v\in B_{n}$,
by the Schwarz inequality (cf. \cite{T}, Proposition I.9.5) we obtain
\begin{align}
\label{cazcsss}
|\t_n(A_Y^{r}D_{YX}&A_X^{s})|\leq\|A_Y\|^{r}\t_n((A_X^*)^{s}D_{YX}^2A_X^{s})^{1/2}\nn\\
\leq&\a(n,s,Q)
\|A_Y\|^{r}\|A_X\|^{s}\t_{n+s}(D_{YX}^2)^{1/2}
\end{align}
where
$$
\a(n,s,Q):=
\begin{cases}
\frac{2(n+s)+1}{2n+1}\,,&Q=2\,,\\
\frac{Q(Q-1)^{n+s}-2}{Q(Q-1)^n-2}\,,&Q>2\,,
\end{cases} 
$$
converges to $(Q-1)^s$ as $n$ increases.\footnote{We are indebted to the referee who pointed out \eqref{rrefe}.} By taking into account \eqref{rrefe} and \eqref{cazcsss}, we get
$$
|\t_n(A_Y^k-A_X^k)|\approx\sum_{l=1}^{k}(Q-1)^{l-1}\|A_Y\|^{k-l}\|A_X\|^{l-1}
\t_{n+l-1}(D_{YX}^2)^{1/2}\,,
$$
which goes to $0$ thanks to \eqref{estadc}. This leads to \eqref{cazcs} for each polynomial 
in $A_X$ and $A_Y$. The proof follows by the Weierstrass Density Theorem and a standard approximation argument.
\end{proof}
As the adjacency operator has non negative entries, we have
$\|A_Y\|\geq\|A_X\|$ under general additive perturbations. The most
interesting case for the physical applications is when the additive
perturbations are negligible.
Put $\d:=\|A_X\|-\|A_Y\|$. It has a very precise physical meaning as an effective chemical potential (cf. \cite{FGI1}, Proposition 7.1). In the case of additive negligible perturbations, we get
\begin{Cor}
Let $F_X:=N_{\|A_X\|\idd-A_X}$, $F_Y:=N_{\|A_Y\|\idd-A_Y}$. We have
\begin{equation}
\label{6}
F_Y(x)=F_X(x+\d)\,.
\end{equation}
\end{Cor}
\begin{proof}
By taking into account the definition of the integrated density of
the states and Proposition \ref{density0}, we get the following:
\begin{align*}
&\int f(x)\di F_Y(x)=\t(f(\|A_Y\|\idd-A_Y))=\t(f(\|A_Y\|\idd-A_X))\\
=&\t(f(\|A_X\|\idd-A_X-\d\idd))=\int f(x-\d)\di F_X(x)=\int f(x)\di F_X(x+\d)\,.
\end{align*}
This leads to \eqref{6}.
\end{proof}
We end the present section by briefly describing the networks
studied in the present paper. We add self loops on a negligible
quantity of vertices of a fixed homogeneous tree (cf. Fig.
\ref{Fig6}).
 \begin{figure}[ht]
     \centering
     \psfig{file=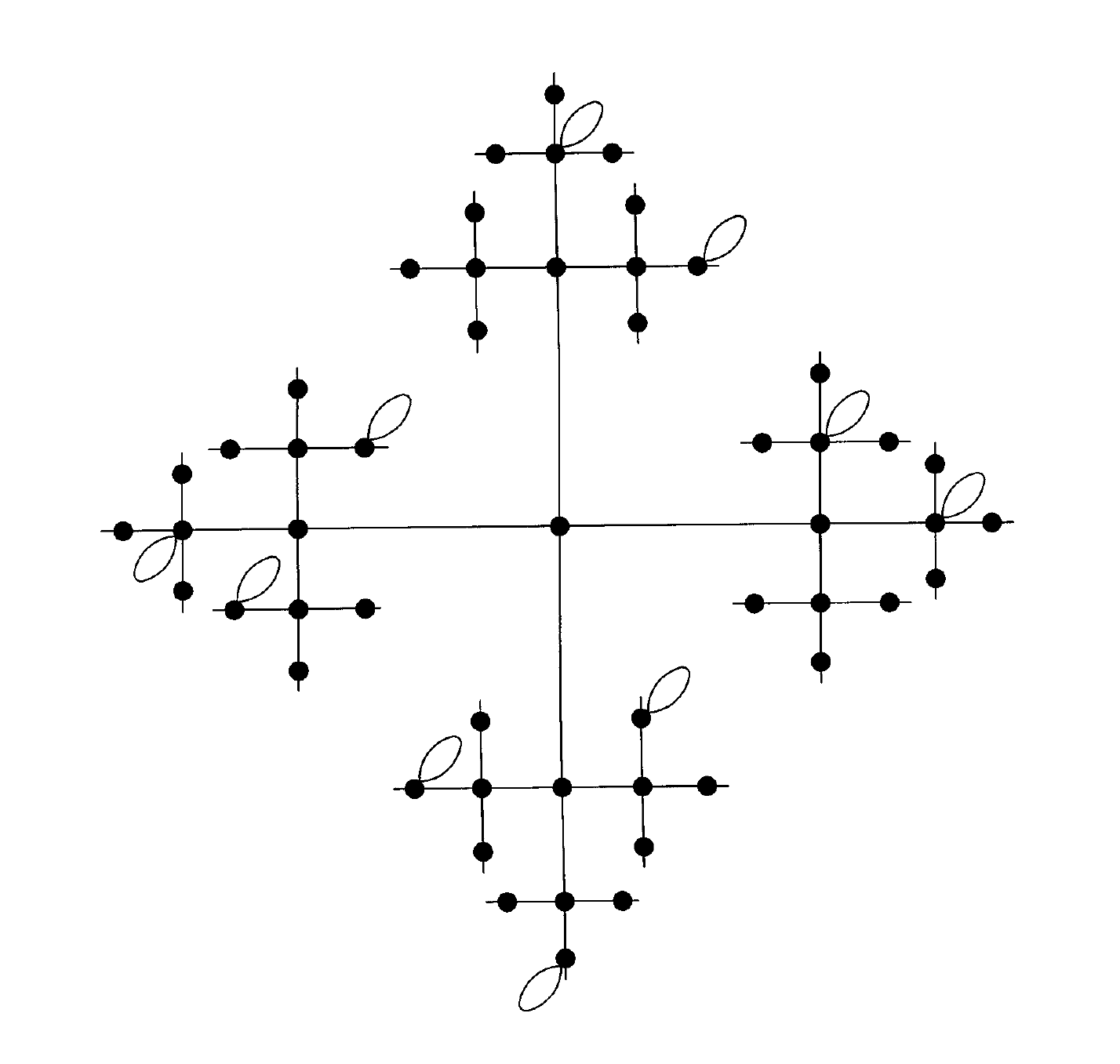,height=2.5in}
     \caption{The perturbation of the Cayley Tree of degree 4 by self loops.}
     \label{Fig6}
     \end{figure}
On one hand, the mathematical analysis becomes simpler as it will
become clear below. On the other hand, as explained in \cite{FGI1},
it is expected that our simplified model captures all the
qualitative phenomena appearing in more complicated examples
relative to general additive negligible perturbations.

\section{the norm of the adjacency operator for perturbed graphs}
\label{sec1}

We start with the homogeneous Cayley Tree $\bg^Q$ of order $Q$,
together with a {\it root} $0\in\bg^Q$ kept fixed during the
analysis. The ball $\bg^Q_n\subset\bg^Q$ is the subgraph made of
vertices and edges at the distance $n$ from the root $0$. The non
homogeneous graphs we deal with are obtained by adding a loop on any
vertex of a subgraph isomorphic to a tree of order $q$ with $1<
q\leq Q$. Another situation is when we add self loops along a sub
path isomorphic to $\bn$ starting from the root. We denote such
graphs by $\bg^{Q,q}$ and $\bh^Q$, respectively, see Fig. \ref{Fig8}
and Fig. \ref{Fig13}. By an abuse of the notation, we write simply
$X$ for the set $VX$ of the vertices of the graph $X$ when this
causes no confusion.

It is not difficult to show that $\bg^{Q,q}$ and $\bh^Q$ are
negligible perturbation of $\bg^Q$, provided $q<Q$. The case
$\bg^{Q,q}$ easily follows from
$$
\frac{|\{e_{xy}\in EX\triangle EY\mid x\in VX_n\}|}{|V\bg^{Q}_n|}
=2\frac{|V\bg^{q}_n|}{|V\bg^{Q}_n|}
=2\frac{(Q-2)[q(q-1)^n-2]}{(q-2)[Q(Q-1)^n-2]}\approx(\frac{q}{Q})^n\,,
$$
whereas the case $\bh^Q$ is analogous to $\bg^{Q,2}$.

The first step is to compute the norm of $\bg^{Q,q}$ by the
results in Lemma \ref{fgiii}. To this end, we consider the more
general situation described as follows. Let $S\subset\bg^{Q}$,
together with $S_n:=S\cap\bg^{Q}_n$. Add a loop to each site of $S$.
Denote by $Y$ and $Y_n$ the graphs obtained by adding self loops on
the sites of $S$ and $S_n$, respectively. Thus, if $S$ is any
subtree of order $q$, $Y=\bg^{Q,q}$, and $Y_n$ is $\bg^{Q}$
perturbed only along the finite subtree $\bg^q_n$, see Fig.
\ref{Fig8}. Define for $\l>\|A_{\bg^{Q}}\|$,
$f_n(\l):=\|P_{\ell^2(S_n)}R_{A_{\bg^{Q}}}(\l)P_{\ell^2(S_n)}\|$ and
$f(\l):=\|P_{\ell^2(S)} R_{A_{\bg^{Q}}}(\l)P_{\ell^2(S)}\|$.
\begin{Lemma}
\label{00}
With the above notations, we get
\begin{itemize}
\item[(i)] $f_n(\l)\uparrow f(\l)$,
\item[(ii)] $\l<\m\Longrightarrow f(\l)>f(\m)$, $f_n(\l)>f_n(\m)$, $n=0,1,2,\dots$\,.
\end{itemize}
\end{Lemma}
\begin{proof}
(i) It follows from  Theorem 6.8  in \cite{S}, as
$P_{\ell^2(S)}R_{A_{\bg^{Q}}}P_{\ell^2(S)}$ has positive entries.

(ii) Let $A$ be a selfadjoint operator and $P$ a selfadjoint projection, both acting on a Hilbert space $\ch$. Let $\|A\|<\l\leq x\leq\m$ and $v\in\ch$ be a unit vector. We obtain by the first identity of the resolvent
$$
\frac{\di\,\,\,}{\di x}\langle R_A(x)v,v\rangle=-\|R_A(x)v\|^2\leq-c\,,
$$
where $c:=\inf_{x\in[\l,\m]}\|x\idd-A\|^{-1}>0$. By integrating both members from $\l$ to $\m$, and taking the supremum on all the unit vectors $v\in P\ch$, we get
$$
f(\m)<f(\l)-c(\m-\l)\leq f(\l)\,.
$$
The assertion follows by putting $A=A_{\bg^{Q}}$ and $P=P_S$. The proof for the $f_n$ is analogous.
\end{proof}
The main object for the analysis of the spectral properties of the resolvent is the secular equation which, for the cases under consideration,  is described in the following
\begin{Thm}
\label{es}
For $\l>\|A_{\bg^{Q}}\|$, the equation
\begin{equation}
\label{sc}
\|P_{\ell^2(S)} R_{A_{\bg^{Q}}}(\l)P_{\ell^2(S)}\|=1
\end{equation}
has at most one solution. If this is the case, the unique solution of \eqref{sc} is the norm $\|A_Y\|$
of $A_Y$, where $Y$ is the perturbation of  $\bg^{Q}$ previously defined. Conversely, if
$\l_{*}:=\|A_Y\|>\|A_{\bg^{Q}}\|$, then $\l_*$ fulfills \eqref{sc}.
\end{Thm}
\begin{proof}
By Lemma \ref{00}, \eqref{sc} has one solution, necessarily unique, if and only if
$\lim_{\l\downarrow\|A_{\bg^{Q}}\|}f(\l)>1$ as $f(\l)$ decreases. Suppose that this is the case.
Again by Lemma \ref{00}, there exists $N$ and, for each $n>N$, a unique
$\l_n>\|A_{\bg^{Q}}\|$ such that $f_n(\l_n)=1$. By Lemma \ref{fgiii} we get
$\|A_Y\|\geq\|A_{Y_n}\|>\|A_{\bg^{Q}}\|$. Suppose now that
$\l_*:=\|A_Y\|>\|A_{\bg^{Q}}\|$, $\l_n:=\|A_{Y_n}\|\uparrow\l_*$ and, again by
Lemma \ref{fgiii}, $f_n(\l_n)=1$ for all the $n$ large enough. By taking into account the Dini Theorem (cf. \cite{Ro}, Theorem 9.11), and Lemma \ref{00}, we get $f(\l_*)=\lim_nf_n(\l_n)=1$, and the proof follows.
\end{proof}
As a useful consequence, we get
\begin{Cor}
With the above notations, if $\|A_Y\|>\|A_{\bg^{Q}}\|$, then
$\{\l\in\bc\mid|\l|>\|A_Y\|\}\subset{\rm P}(A_Y)$, and
\begin{equation}
\label{koro}
R_{A_Y}(\l)=R_{A_{\bg^{Q}}}(\l)
+ R_{A_{\bg^{Q}}}(\l)\left(\idd_{\ell^2(S)}-S(\l)\right)^{-1}P_{\ell^2(S)}R_{A_{\bg^{Q}}}(\l)\,.
\end{equation}
\end{Cor}
\begin{proof}
Suppose that, for some $\l\in\bc$, $\idd_{\ell^2(S)}-S(\l)$ is invertible in $\cb(\ell^2(S))$. Then
a straightforward calculation shows that the operator on the l.h.s. of \eqref{koro} provides the left and right inverse of $\l\idd_{\ell^2(\bg^{Q})}-A_Y$. Namely, it leads to the resolvent of $A_Y$ for such a $\l$. The proof follows by Theorem \ref{es} as $\idd_{\ell^2(S)}-S(\l)$ is invertible for those
$\l\in\bc$ such that $|\l|>\|A_Y\|$.
\end{proof}
The main cases of interest in the present paper are $S\sim\bg^{q}$, $1<q\leq Q$. Then \eqref{sc} becomes
$$
\|P_{\ell^2(\bg^{q})}R_{A_{\bg^{Q}}}(\l)P_{\ell^2(\bg^{q})}\|=1\,,
$$
and allows us to determine whether $\|A_{\bg^{Q,q}}\|>\|A_{\bg^{Q}}\|$. Another case of interest is
$S\sim\bn$. When $q=2$, $S\sim\bz$. Thus, the secular equation \eqref{sc} allows us to study the situation when $S\sim\bn$ as well, see \eqref{loo1}.

Let $d$ be the standard distance on the Cayley Tree of order $Q$. The following walk generating function
\begin{equation}
\label{ww}
W_{x,y}(\xi)=\bigg(\frac{1-\sqrt{1-4(Q-1)\xi^2}}{2(Q-1)\xi}\bigg)^{d(x,y)}\frac{2(Q-1)}
{Q-2+Q\sqrt{1-4(Q-1)\xi^2}}\,,
\end{equation}
reported in Section 7.D of \cite{MW}, is crucial for our computations, see below.

To have an idea of what is happening, we consider the following
simple example $X$ obtained by perturbing $\bg^{Q}$ by a self loop
based on the root $0$, see Fig \ref{Fig11}.
 \begin{figure}[ht]
     \centering
     \psfig{file=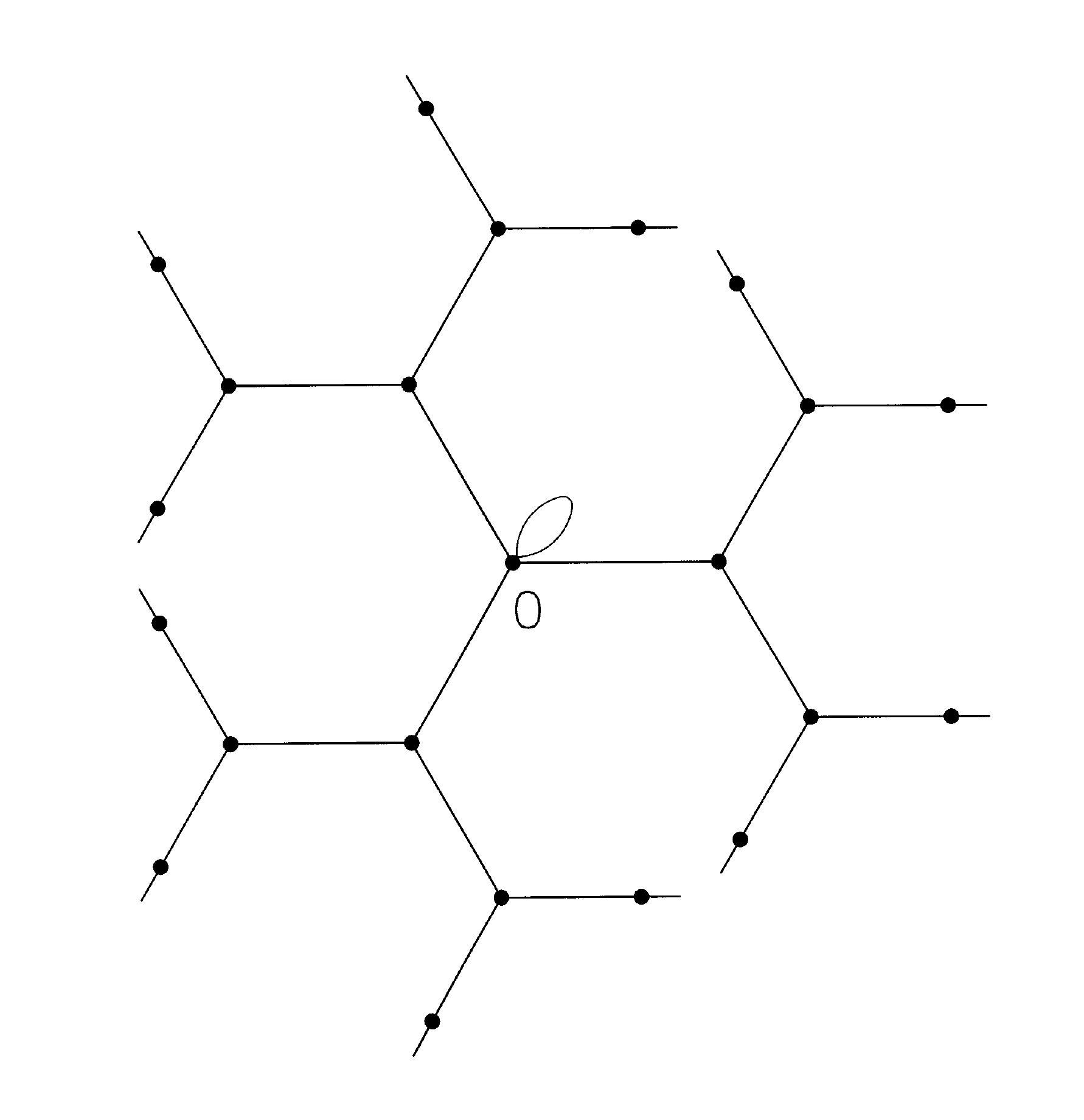,height=2.2in} \quad \psfig{file=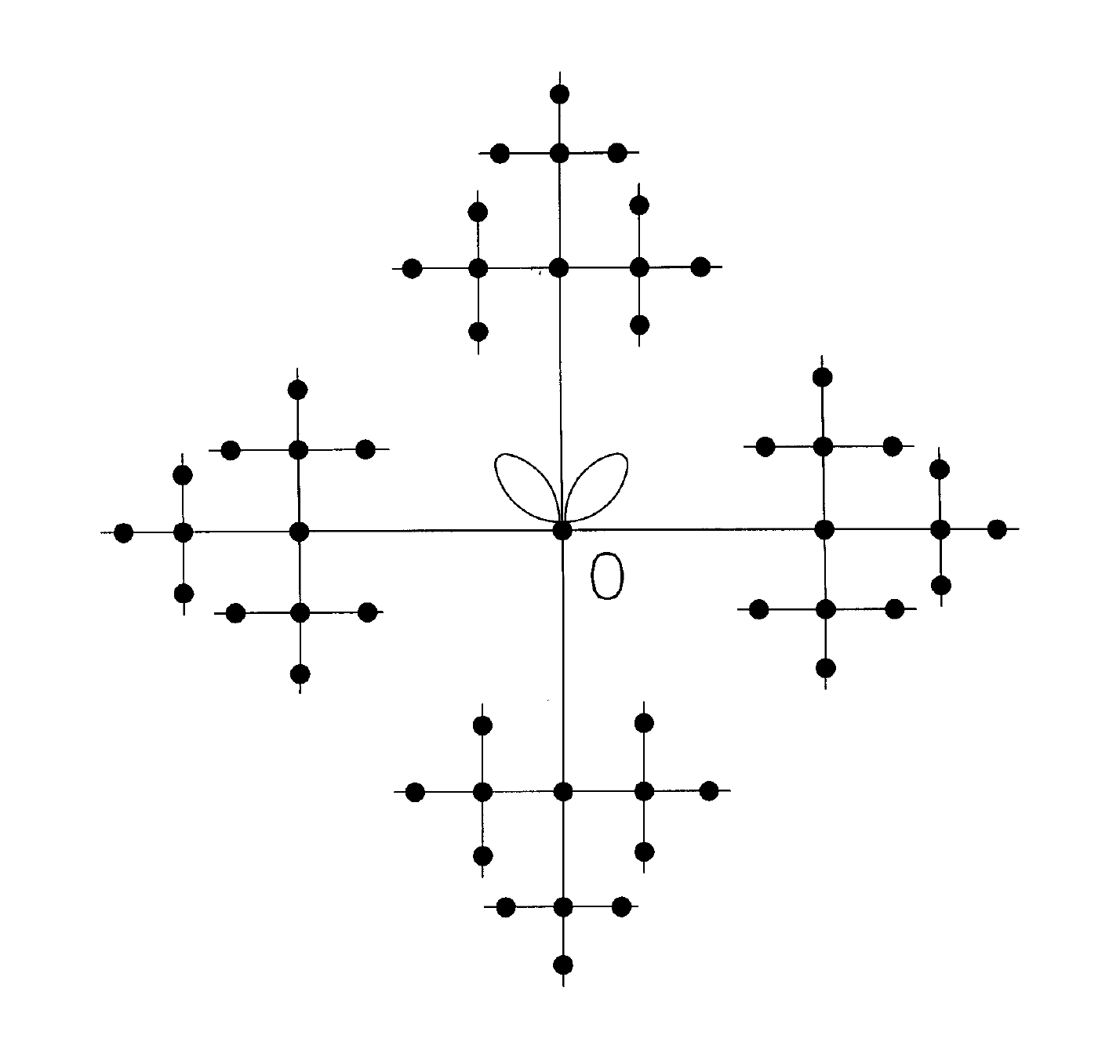,height=2.5in}
     \caption{The perturbation of $\bg^3$ and $\bg^4$ by self loops.}
     \label{Fig11}
     \end{figure}
With $\xi:=1/\l$, the secular equation
\eqref{sc} becomes
\begin{equation}
\label{cons1}
\xi W_{0,0}(\xi)=1
\end{equation}
by (7.6) of \cite{MW}. By taking into account
$$
\frac{Q-2}{2(Q-1)}<\xi<\frac{1}{2\sqrt{Q-1}}\,,
$$
we show that \eqref{cons1}, has the solution, necessarily unique,
$$
\xi=\frac{(1+\sqrt{5})Q-2}{2(Q^2+Q-1)}
$$
only if $Q=3$.\footnote{The case when we add a loop to the root of $\bg^2\sim\bz$ is already treated in \cite{FGI1}, Section 8.} In this case,
$$
\|A_X\|=\frac{22}{1+3\sqrt{5}}>\|A_{\bg^{Q}}\|=2\sqrt2\,.
$$
Other objects of interest are the adjacency $A_X$ and its Perron Frobenius eigenvector $v$ for the perturbed graph $X$. We get by \eqref{koro}, \eqref{ww},
$$
R_{A_X}(\l)
=R_{A_{\bg^{3}}}(\l)\bigg(\idd+\bigg(1-\frac1{\l}W_{0,0}\bigg(\frac1{\l}\bigg)\bigg)^{-1}
P_{\d_0}R_{A_{\bg^{3}}}(\l)\bigg)\,,
$$
where
$$
W_{0,0}(\xi)=\frac4{1+3\sqrt{1-8\xi^2}}\,.
$$
By \eqref{zeta}, we have for the Perron Frobenius eigenvector,
$$
v=R_{A_{\bg^{Q}}}(\|A_{X}\|)\d_0\,.
$$
As $v\in\ell^2(X)$, $A_X$ is recurrent. This implies by \cite{S}, Theorem 6.2, that $v$ is the unique (up to a multiplicative scalar) Perron Frobenius eigenvector for $A_X$.

In the cases $Q>3$ the secular equation \eqref{cons1} has no solution greater than $2\sqrt{Q-1}$. Then $\|A_X\|=\|A_{\bg^{Q}}\|$, that is,
the perturbation is too small to change the norm of the adjacency
operator, and to create an hidden zone of the spectrum near zero of
the Hamiltonian. In these cases, it is easy to show that it is
enough to add a large enough number of self loops on the chosen root
(cf. Fig. \ref{Fig11}) in order to increase the norm of the
adjacency and then to obtain the hidden spectrum. Indeed, by the
same computations as in Proposition \ref{norm}, we get
$$
n=\left[\frac{Q-2}{\sqrt{Q-1}}\right]+1\,,
$$
where $n$ is the minimum number of the self loops to add to the tree of order $Q$ to have the hidden spectrum.
This means that, in order to obtain the hidden spectrum also for $4\leq Q\leq6$, it is enough to add just two loops to the chosen root of $\bg^Q$, and so on.

We end the present section by remarking the following very
surprising facts. It is enough to add a small number of edges to
$\bg^{Q}$ in order to change dramatically the spectral properties
near $\|A_X\|$, of the adjacency operator $A_X$ of the perturbed
graph $X$, even if the graph under consideration grows exponentially. For example, the perturbed adjacency could exhibit hidden
spectrum. In addition, it could become recurrent and finally the
shape of the Perron Frobenius eigenvector changes dramatically. As
explained above and in accordance with the results in \cite{FGI1},
the perturbed network could exhibit very different properties
compared with the original unperturbed one.

\section{the perturbed trees along a subtree isomorphic to $\bz$}
\label{z}

The present section is devoted to the network $\bg^{Q,2}$ obtained by perturbing a homogeneous tree of order $Q$ along a path isomorphic to $\bz$, see Fig. \ref{Fig13}. In this case, the subset 
$S$ appearing in Theorem \ref{es} is nothing but $\bg^2\sim\bz$. We simply write 
$\bz=S\subset\bg^{Q,2}$.
 \begin{figure}[ht]
     \centering
     \psfig{file=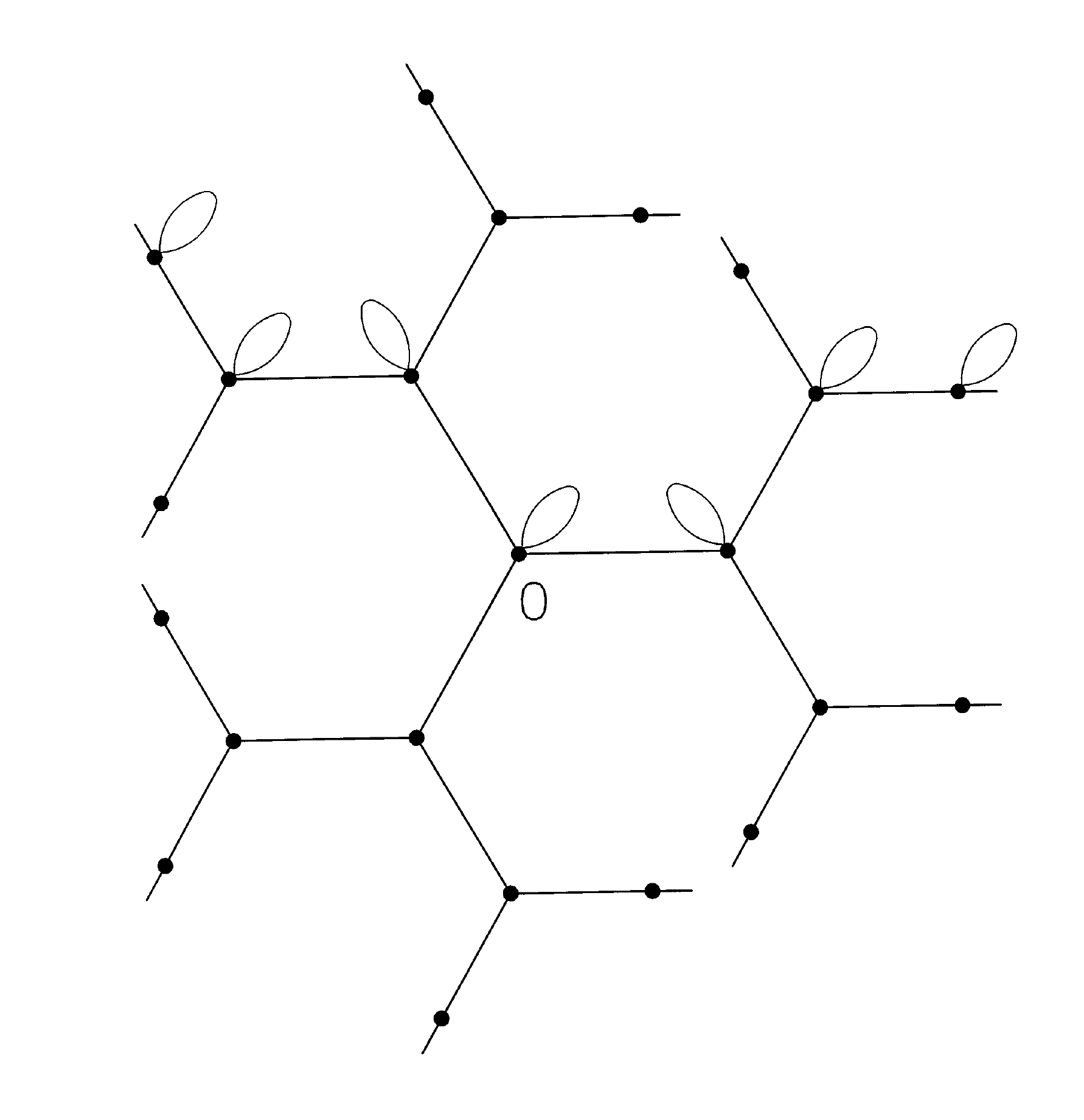,height=2.2in} \quad \psfig{file=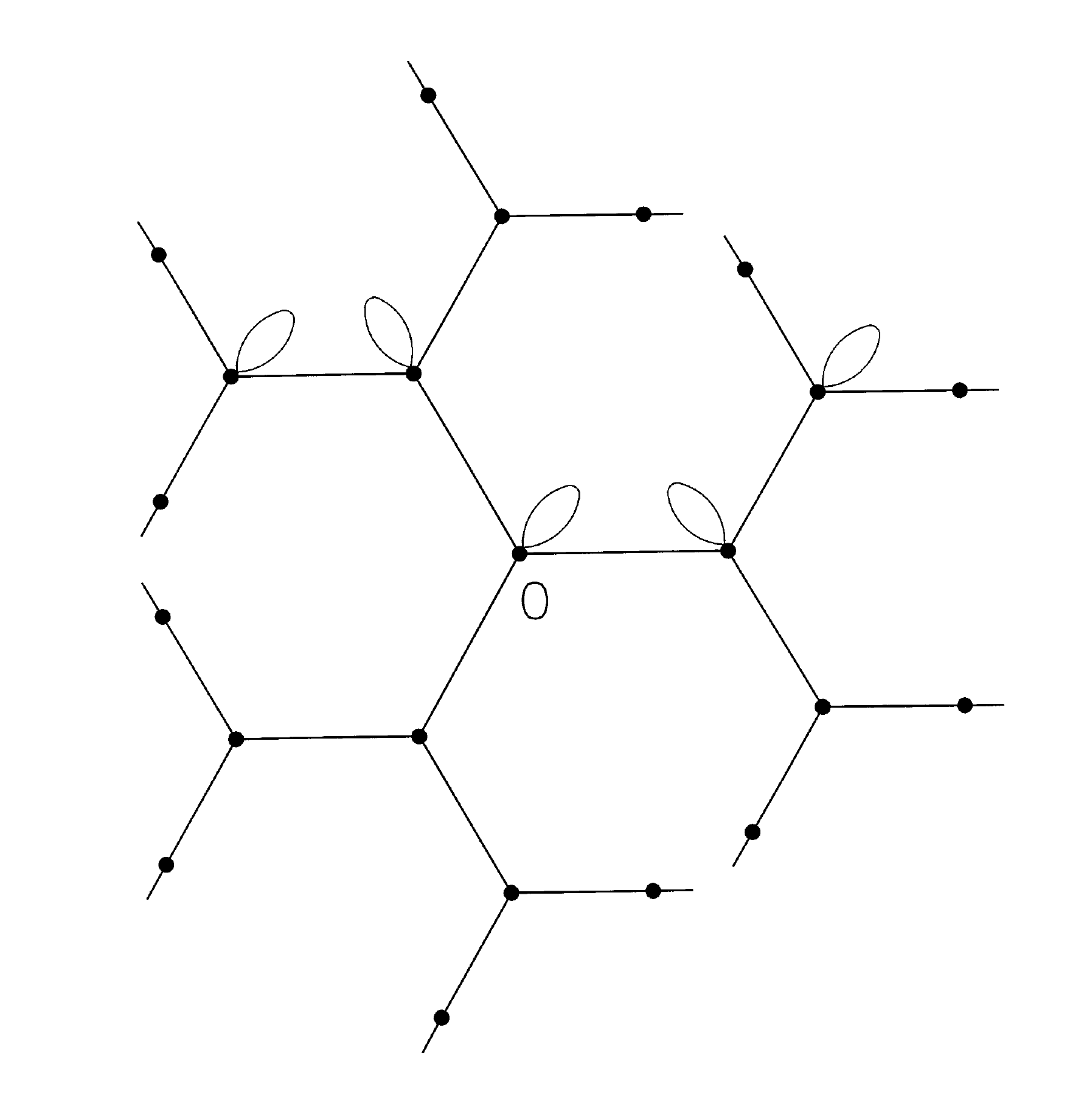,height=2.2in}
     \caption{The networks $\bg^{3,2}$ and $X_2$.}
     \label{Fig13}
     \end{figure}
To this end, we consider for $a<1$, the operator $T_a$ acting on $\ell^2(\bz)$, and defined as
\begin{equation*}
T_av:=f_a*v\,.
\end{equation*}
Here, $f_a(x):=a^{d(x,0)}$, $d$ being the standard distance on the tree $\bg^{Q}$, and $0$ any fixed root on $\bz\subset\bg^{Q}$. By using the Fourier transform, $\widehat{T_a}$ becomes the multiplication operator on
$L^2\big(\bt,\frac{\di\th}{2\pi}\big)$ by the Poisson kernel
\begin{equation}
\label{cov1}
P_a(e^{\imath\th})=\frac{1-a^2}{1-a(e^{\imath\th}+e^{-\imath\th})+a^2}\equiv\frac{1-a^2}{1-2a\cos\th+a^2}\,,
\end{equation}
see e.g. \cite{R}, Section 11.2.
Denote
\begin{equation}
\label{000}
a(\l):=\frac{1-\sqrt{1-\frac{4(Q-1)}
{\l^2}}}{\frac{2(Q-1)}{\l}}\,,
\end{equation}
\begin{equation}
\label{333}
\m(\l):=\frac{Q-2+Q\sqrt{1-\frac{4(Q-1)}{\l^2}}}
{\frac{2(Q-1)}{\l}}\,.
\end{equation}
By taking into account (7.6) in \cite{MW}, the secular equation \eqref{sc} becomes in this case
$\|T_{a(\l)}\|=\m(\l)$. This means by \eqref{cov1},
\begin{equation}
\label{1aa}
\frac{1+a(\l)}{1-a(\l)}=\m(\l)
\end{equation}
as $\|T_{a(\l)}\|=P_{a(\l)}(1)$. Consider $\l>\|A_{\bg^{Q}}\|\equiv2\sqrt{Q-1}$. In such a range, the secular equation \eqref{1aa} has no solution if $Q>7$. On the other hand, if $2\leq Q\leq7$,
\eqref{1aa} has a solution (necessarily unique) given by
\begin{equation}
\label{11aa}
\l=\frac{3-\sqrt{5}}{2}Q+\sqrt{5}\,.
\end{equation}
The reader is referred to Proposition \ref{norm} for the general case $2\leq q\leq Q$.

The main properties of the resolvent of the adjacency of $\bg^{Q,2}$
useful in the sequel, are summarized in the following
\begin{Thm}
\label{ress}
Let $3\leq Q\leq7$, and $\l>\l_*\equiv\|A_{\bg^{Q,2}}\|$ given in \eqref{11aa}. Then we have
\begin{equation}
\label{ress1}
R_{A_{\bg^{Q,2}}}(\l)=R_{A_{\bg^{Q}}}(\l)\bigg[\idd_{\ell^2(\bg^{Q})} +P_{\ell^2(\bz)}
\bigg(\idd_{\ell^2(\bz)}-\frac1{\l}W\bigg(\frac1{\l}\bigg)\bigg)^{-1}P_{\ell^2(\bz)}R_{A_{\bg^{Q}}}(\l)\bigg]\,,
\end{equation}
where $W$ is the operator acting on $\ell^2(\bz)$ whose matrix elements are given by \eqref{ww}.
In addition, $\bg^{Q,2}$ is recurrent.
\end{Thm}
\begin{proof}
We have shown that the secular equation \eqref{1aa} has a solution $\l_*$ (necessarily unique)
greater then $2\sqrt{Q-1}$, only if $Q\leq7$. For such cases, $\l_*$ is precisely the norm of the adjacency $A_{\bg^{Q,2}}$ of the perturbed graph. In addition, Lemma \ref{00} and Theorem
\ref{es} show that
$$
\idd_{\ell^2(\bz)}-S(\l):=P_{\ell^2(\bz)}-P_{\ell^2(\bz)}R_{A_{\bg^{Q}}}P_{\ell^2(\bz)}
$$
acting on $\ell^2(\bz)$, is invertible, provided $\l>\l_*$.
By taking account \eqref{koro} and \eqref{ww},
\eqref{ress1} gives rise the resolvent of $A_{\bg^{Q,2}}$ for such positive $\l$.

As $\l_*\equiv\|A_{\bg^{Q,2}}\|>\|A_{\bg^{Q}}\|$ if $Q\leq7$, to check the recurrence it is enough (cf. \cite{S})  to study the limit as $\l\downarrow\l_*$ of
\begin{align*}
\big\langle R_{A_{\bg^{Q,2}}}(\l)\d_0,\d_0\big\rangle&
=\big\langle S(\l)\big(\idd_{\bz}-S(\l)\big)^{-1}\d_0,\d_0\big\rangle\\
=&\frac1{2\pi}\int_0^{2\pi}\frac{P_{a(\l)}(e^{\imath\th})}{\m(\l)-P_{a(\l)}(e^{\imath\th})}\di\th\,,
\end{align*}
where $R_{A_{\bg^{Q,2}}}$ is given in \eqref{ress1}, and $\m(\l)$, $a(\l)$ are given by
\eqref{333} and \eqref{000}, respectively. We write $a$, $\m$ for $a(\l)$, $\m(\l)$, $\l\geq\l_*$, respectively.

We pass to the complex plane by using the analytic continuation of $P_{a}$, getting
\begin{equation}
\label{ress2}
\big\langle R_{A_{\bg^{Q,2}}}(\l)\d_0,\d_0\big\rangle
=\frac{a^2-1}{2\pi\imath}\oint\frac{\di z}
{a\m z^2-[(1+a^2)\m-(1-a^2)]z+a\m}\,,
\end{equation}
where both $a$ and $\m$ are functions of $\l$ as before. It is straightforward to show that if
$\l>\l_*$ then $\m>\frac{1+a}{1-a}$. In addition, $\l>\l_*$ implies
$$
\D:=[(1+a^2)\m-(1-a^2)]^2-4\m^2a^2>0\,.
$$
The last is zero if $\l=\l_*$, or equivalently if $\m=\frac{1+a}{1-a}$. If $\l>\l_*$ is sufficiently close to $\l_*$, \eqref{ress2} becomes
$$
\big\langle R_{A_{\bg^{Q,2}}}(\l)\d_0,\d_0\big\rangle=\frac{a^2-1}{a\m}\frac1{2\pi\imath}
\oint\frac{\di z}{(z-z_+)(z-z_-)}\,,
$$
where
\begin{equation}
\label{zed}
z_{\pm}:=\frac{(1+a^2)\m-(1-a^2)\pm\sqrt{\D}}{2a\m}
\end{equation}
are close to $1$, with $z_-<1<z_+$. Thus,
$$
\big\langle R_{A_{\bg^{Q,2}}}(\l)\d_0,\d_0\big\rangle=\frac{1-a^2}{a\m(z_+-z_-)}
=\frac{1-a^2}{\sqrt{\D}}\to+\infty
$$
if $\l\downarrow\l_*$, that is $R_{A_{\bg^{Q,2}}}$ is recurrent.
\end{proof}

We end the present section by describing the (generalized)
Perron--Frobenius eigenvector on $\bg^{Q,2}$. As $P_a$ is recurrent,
the uniform weight $v:=1$ identically on $\bz$, is the unique (up to
a constant), Perron--Frobenius eigenvector of $T_a$.
\begin{Lemma}
\label{loo}
Let $S$ be a connected subgraph of $\bg^{Q}$ and $x\in\bg^{Q}$. Then there exist a unique $y(x)\in S$ such that $d(x,S)=d(x,y(x))$.
\end{Lemma}
\begin{proof}
By a standard compactness argument, $d(x,y)$, $y\in S$ attains its
minimum. Suppose that such a minimum is not unique. As $S$ is connected, there exists a loop in the tree $\bg^{Q}$ which is a contradiction.
\end{proof}
Let now $v_n$ be the normalizable Perron Frobenius eigenvector
 of the adjacency operator of the graph $X_n$ (cf. Fig. \ref{Fig13})
obtained by perturbing $\bg^{Q}$ along a segment $S_n$ made of $2n+1$ points centered in the root $0$, which exists by Lemma \ref{fgiii}. Normalize such a vector by putting 
$v_n(0)=1$.
\begin{Thm}
If $Q\leq 7$ then the Perron Frobenius eigenvector for
$A_{\bg^{Q,2}}$ is unique up to a multiplicative constant, and it is given by a multiple of
\begin{equation}
\label{pfz}
v(x)=a(\l_*)^{d(x,\bz)}\,,
\end{equation}
where $a(\l_*)$ is given by \eqref{000} and $\l_*$ fulfills \eqref{11aa} with $S=\bz$.

With the above notations, $v$ is the pointwise limit of the Perron Frobenius eigenvectors $v_n$.
\end{Thm}
\begin{proof}
As $A_{\bg^{Q,2}}$ is recurrent (cf. Theorem \ref{ress}), the Perron Frobenius eigenvector
is unique, see \cite{S}, Theorem 6.2. Let $v$ be such a Perron Frobenius eigenvector, normalized as $v(0)=1$. By \eqref{zeta}, \eqref{ww}, and Lemma
\ref{loo}, the Perron Frobenius eigenvector $v_n$ described above is given by
$$
v_n(x):=a(\l_n)^{d(x,y_n(x))}w_n(y_n(x)).
$$
Here, $a(\l_n)$ is given by \eqref{000}, $\l_n$ fulfills the secular equation \eqref{sc} with $S=S_n$, and finally $w_n$ is the Perron Frobenius eigenvector of
$P_{\ell^2(S_n)}T_{a(\l_n)}P_{\ell^2(S_n)}$, extended to $0$ outside $S_n\subset\bz$ and normalized such that
$w_n(0)=1$. As $d(x,y_n(x))$ converges pointwise to $d(x,\bz)$ and $\bg^{Q,2}$ is recurrent (which implies $w_n(x)$ converges pointwise to $v(x)$ whenever $x\in S\sim\bz$), it is enough to show that
$\lim_nw_n(x)=1$ pointwise for $x$ in the subgraph $S$ of $\bg^{Q}$ isomorphic to $\bz$, supporting the perturbation.\footnote{See Theorem \ref{pftr} for an alternative proof of this part.} By the Fatou Lemma we get for $x\in S$,
\begin{align*}
v(x)=\lim_nw_n(x)=\lim_n\bigg(&\|T_{a(\l_n)}\|^{-1}\sum_{|y|\leq n}[T_{a(\l_n)}]_{x,y}w_n(y)\bigg)\\
\geq&\frac{1-a(\l_*)}{1+a(\l_*)}\sum_{y\in\bz}[T_{a(\l_*)}]_{x,y}v(y)\,.
\end{align*}
This means that $v$, restricted to (the subgraph isomorphic to) $\bz$ is a subinvariant weight for
$T_{a(\l_*)}$, which is unique (up to a multiple) and equal to the uniform distribution, see \cite{S}, Theorem 6.2.
\end{proof}

\section{the perturbed trees along a subtree isomorphic to $\bn$}

In the present section we consider the network $\bh^{Q}$ obtained by perturbing a homogeneous tree of order $Q$ along a path isomorphic to $\bn$, see Fig. \ref{Fig14}. In this case $S\sim\bn$. As before, we denote such a subgraph $S$ directly by $\bn$.
 \begin{figure}[ht]
     \centering
     \psfig{file=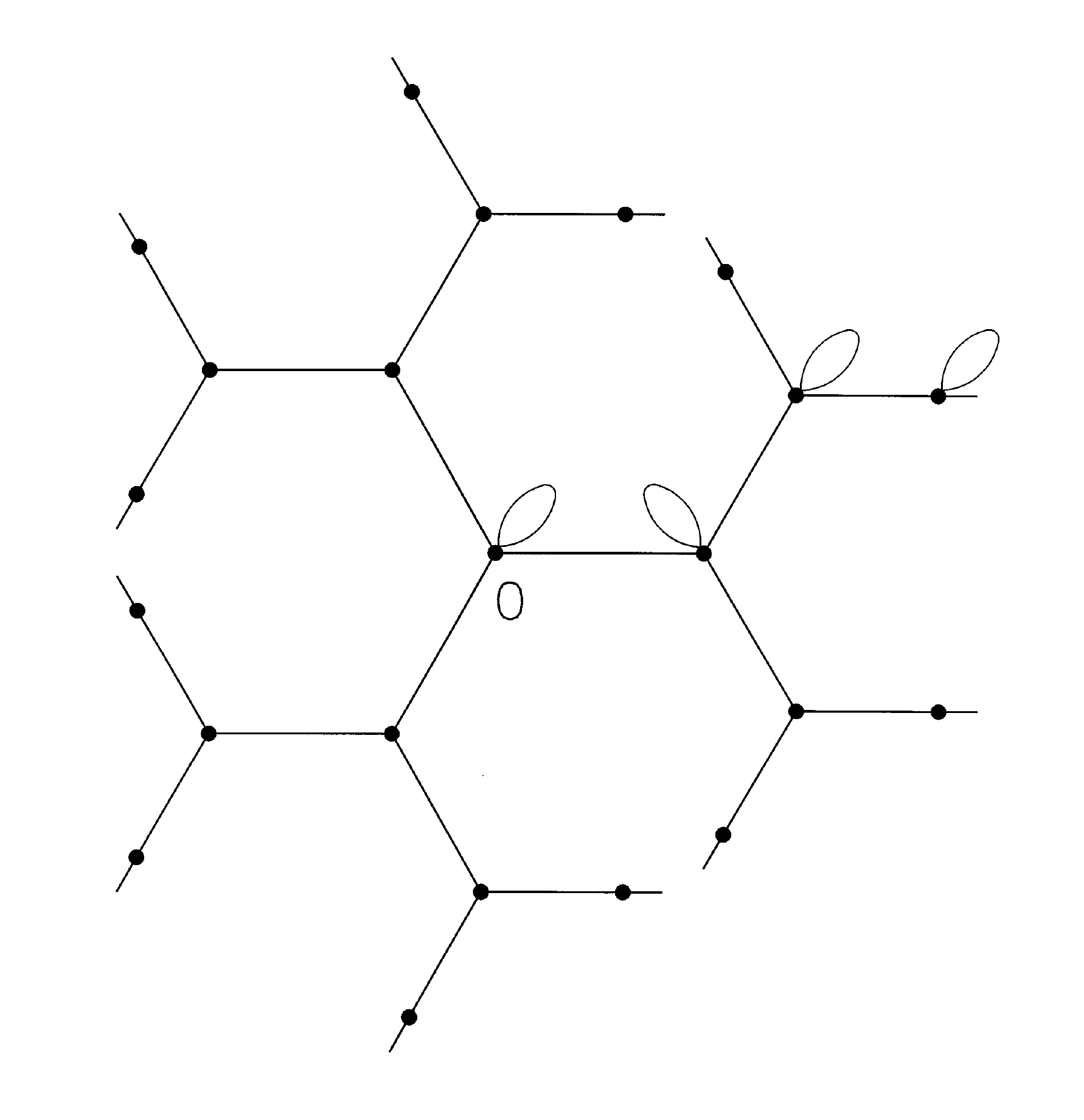,height=2.2in} \quad \psfig{file=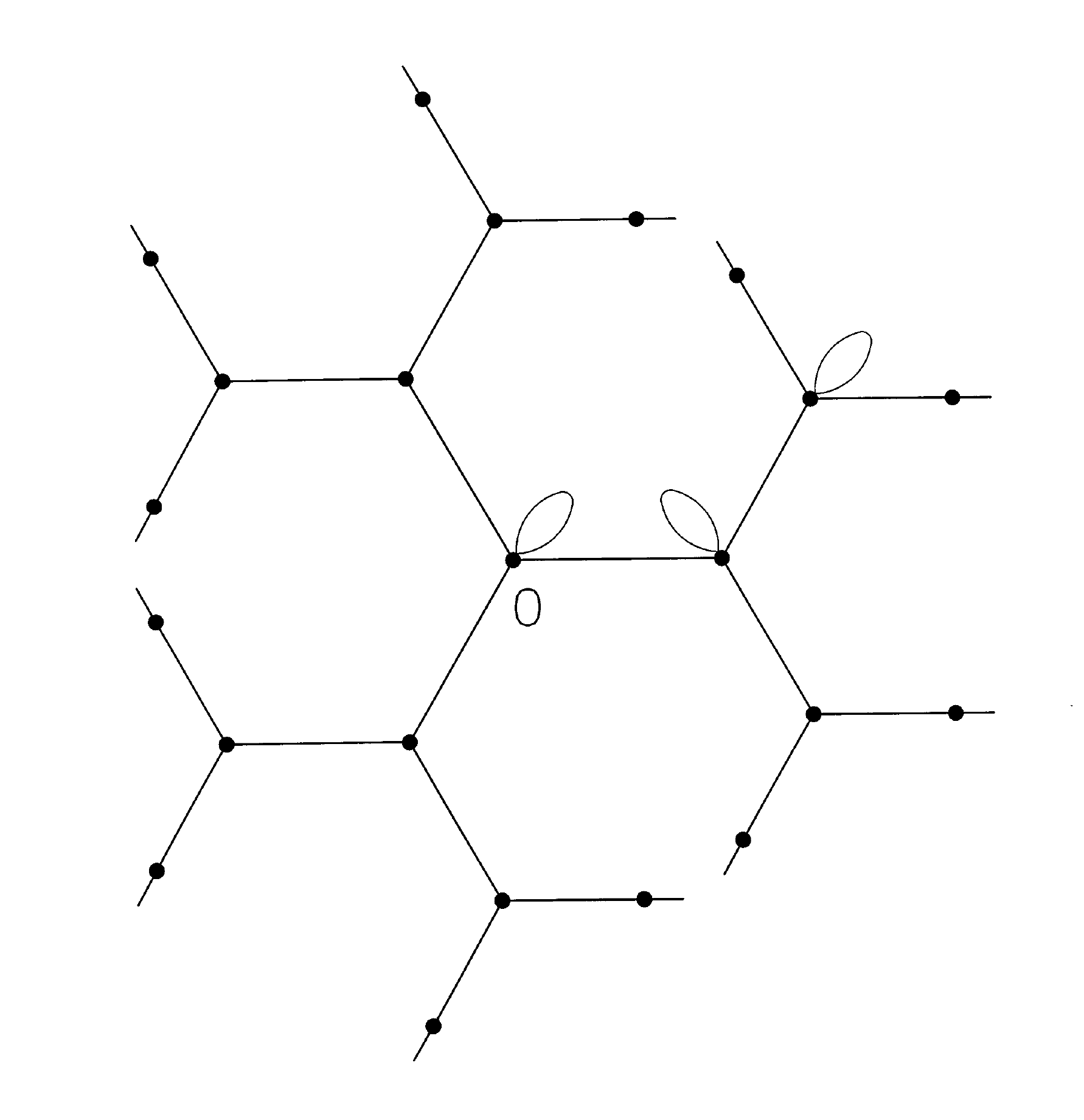,height=2.2in}
     \caption{The networks $\bh^{3}$ and $Y_2$.}
     \label{Fig14}
     \end{figure}

Consider the subgraph
$N_n\subset\bg^{Q}$ made of $n+1$ points starting from the root $0$ of $\bg^{Q}$.
Let $v_n$ be the Perron Frobenius eigenvector of the adjacency $A_{Y_n}$ of the graph
$Y_n\subset\bh^{Q}$ (cf. Fig. \ref{Fig14}), the last
obtained by perturbing $\bg^{Q}$ with self loops along $N_n$, normalized such that
$v_n(0)=1$. We start our analysis with the following
\begin{Lemma}
We have
\begin{equation}
\label{loo1}
\|P_{\ell^2(\bz)}R_{A_{\bg^{Q}}}(\l)P_{\ell^2(\bz)}\|
=\|P_{\ell^2(\bn)}R_{A_{\bg^{Q}}}(\l)P_{\ell^2(\bn)}\|
=\frac{1+a(\l)}{\m(\l)[1-a(\l)]}\,,
\end{equation}
where $a(\l)$ is given by \eqref{000}, and $\m(\l)$ is given by \eqref{333}.
\end{Lemma}
\begin{proof}
Fix $a<1$. We have
$$
P_{\ell^2(S_n)}T_{a}P_{\ell^2(S_n)}=P_{\ell^2(N_{2n})}T_{a}P_{\ell^2(N_{2n})}\,.
$$
Then we get
\begin{align*}
\|P_{\ell^2(\bn)}T_{a}P_{\ell^2(\bn)}\|
=&\lim_n\|P_{\ell^2(N_n)}T_{a}P_{\ell^2(N_n)}\|
=\lim_n\|P_{\ell^2(N_{2n})}T_{a}P_{\ell^2(N_{2n})}\|\\
\equiv&\lim_n\|P_{\ell^2(S_n)}T_{a}P_{\ell^2(S_n)}\|
=\|P_{\ell^2(\bz)}T_{a}P_{\ell^2(\bz)}\|\,.
\end{align*}
\end{proof}
The main properties of the Perron Frobenius eigenvector of $A_{\bh^{Q}}$ are summarized in the following
\begin{Thm}
\label{czz} 
Suppose that $Q\leq 7$. With the above notations, $v_n$
converges pointwise to a weight $v$ which is a Perron Frobenius
eigenvector for $A_{\bh^{Q}}$. It is given by
$$
v(x)=a(\l_*)^{d(x,\bn)}\big[y(x)(1-a(\l_*))+1\big]\,,
$$
where, $y(x)\in\bn$ is described in Lemma \ref{loo},
$a(\l_*)$ is given by \eqref{000}, and $\l_*$ fulfills  \eqref{11aa}.
\end{Thm}
\begin{proof}
We begin by noticing that
$$
v_n(x)=a(\l_n)^{d(x,N_n)}w_n\big(y_n(x)\big)\,,
$$
where $y_n(x)$ is the element of $N_n$ realizing the distance between $x$ and $N_n$
(cf. Lemma \ref{loo}), and $w_n$ is the Perron Frobenius eigenvector of
$P_{\ell^2(N_n)}T_{a(\l_n)}P_{\ell^2(N_n)}$, normalized at the origin of $N_n$ (i.e. $w_n(0)=1$). The result follows if we prove that, for each fixed $k\in\bn$, $w_n(k)$ converges to $(1-a(\l_*))k+1$ as $n$ goes to $\infty$.

Let $\La_n:=\|P_{\ell^2(N_n)}T_{a(\l_n)}P_{\ell^2(N_n)}\|$. As $\m(\l_n)\uparrow\m(\l_*)$,
by Lemma \ref{00} we get
$$
\La_n\uparrow\La_*:=\|P_{\ell^2(\bn)}T_{a(\l_*)}P_{\ell^2(\bn)}\|=\frac{1+a(\l_*)}{1-a(\l_*)}\,.
$$
In addition, we have also $a(\l_n)\downarrow a(\l_*)$. Define $\s_n(k):=a(\l_n)^kw_n(k)$. It is straightforward to see that the solution for the $\s_n(k)$, $0\leq k\leq n$, $n\in\bn$ is given by
\begin{equation}
\label{iiff}
\s_n(k)=1+\frac1{\La_n}\sum_{l=0}^{k-1}\big(a(\l_n)^{2(l-k)}-1\big)\s_n(l)\,,\quad n\in\bn\,.
\end{equation}
Namely, the form of the system defining the $\s_n$ in terms of
$\La_n$, considered as a known quantity, is triangular and independent on the size (i.e. on
$n\in\bn$). By the previous claims, thanks to the fact that
$a(\l_n)\to a(\l_*)$ and $\La_n\to\La_*$ (cf. \eqref{loo1}),
$\s_n(k)$ converges pointwise in $k$ when $n\to\infty$ to
$$
\s(k)=a(\l_*)^k\big[(1-a(\l_*))k+1\big]\,,
$$
which is precisely the limit of \eqref{iiff} as $n\to\infty$.
\end{proof}
Concerning the resolvent of $A_{\bh^Q}$ and the transience character, we get
\begin{Thm}
Suppose that $Q\leq 7$ and $\l>\l_*$ given in \eqref{11aa}. We have
\begin{equation*}
R_{A_{\bh^{Q}}}(\l)=R_{A_{\bg^{Q}}}(\l)\bigg[\idd_{\ell^2(\bg^{Q})} +P_{\ell^2(\bn)}
\bigg(\idd_{\bn}-\frac1{\l}W\bigg(\frac1{\l}\bigg)\bigg)^{-1}P_{\ell^2(\bn)}R_{A_{\bg^{Q}}}(\l)\bigg]\,,
\end{equation*}
where $W$ is the operator acting on $\bn$ given by \eqref{ww}. In addition, $\bh^{Q}$ is transient.
\end{Thm}
\begin{proof}
The proof of the first part follows along the same lines as the
corresponding part of Theorem \ref{ress}. In order to check the
transience, we start by studying the equation
\begin{equation}
\label{har}
\big(\m P_{\ell^2(\bn)}-P_{\ell^2(\bn)}T_aP_{\ell^2(\bn)}\big)v=P_{\ell^2(\bn)}T_aP_{\ell^2(\bn)}\d_0
\end{equation}
where for $\l>\l_*$, $\m=\m(\l)$, $a=a(\l)$ are given by \eqref{333} and \eqref{000}, respectively. By using the Neumann expansion of $\idd_{\ell^2(\bn)}-P_{\ell^2(\bn)}T_aP_{\ell^2(\bn)}/\m$, we argue that $v$ has positive entries. After defining
$$
f(e^{\imath\th}):=\sum_{k\geq0}v(k)e^{\imath k\th}\,,
$$
and denoting $M_g$ the multiplication operator by the function $g$,
\eqref{har} becomes
\begin{equation}
\label{har1}
\big(\m P_{H^2(\bt)}-P_{H^2(\bt)}M_{P_a}P_{H^2(\bt)}\big)f=P_{H^2(\bt)}M_{P_a}P_{H^2(\bt)}1
\end{equation}
where $1$ is the constant function on the unit circle, and $H^2(\bt)\subset L^2(\bt)$ is the Hardy space which is isomorphic to the $L^2$--functions on the unit circle with vanishing Fourier coefficients corresponding to the negative frequences (cf. \cite{R}, Chapter 17). By passing to the conjugates,
\eqref{har1} leads to
\begin{equation}
\label{har2}
\big(\m P_{CH^2(\bt)}-P_{CH^2(\bt)}M_{P_a}P_{CH^2(\bt)}\big)\bar f
=P_{CH^2(\bt)}M_{P_a}P_{CH^2(\bt)}1
\end{equation}
where $M_{P_a}$ is the multiplication operator by the Poisson kernel $P_a(e^{\imath\th})$, $C$ is the canonical conjugation operator acting on functions defined on the circle, with $Cf\equiv\bar f$ given by
$$
\overline{f(e^{\imath\th})}:=\sum_{k\geq0}v(k)e^{-\imath k\th}
$$
as $v$ has positive entries. Define
$$
F(e^{\imath\th}):=\sum_{k\in\bz}v(|k|)e^{\imath k\th}\,,\quad \G:=\sum_{k=1}^{+\infty}v(k)a^{k}\,.
$$
We now compute
\begin{align}
\label{mpaf}
&M_{P_a}F=\sum_{k,l}a^{|k-l|}v(|l|)e^{\imath k\th}
=\sum_{k,l\geq0}a^{|k-l|}v(|l|)e^{\imath k\th}\nn\\
+&\sum_{k,l\leq0}a^{|k-l|}v(|l|)e^{\imath k\th}-v(0)
+\sum_{k,l>0}a^{k+l}v(l)\big(e^{\imath k\th}+e^{-\imath k\th}\big)\\
=&P_{H^2(\bt)}M_{P_a}P_{H^2(\bt)}f+P_{CH^2(\bt)}M_{P_a}P_{CH^2(\bt)}\bar f-v(0)
+(\G-1)P_a\,.\nn
\end{align}
By taking into account \eqref{har1}, \eqref{har2} and \eqref{mpaf}, we obtain
\begin{equation*}
\big(\m\idd_{L^2(\bt)}-P_a\big)F=(1-\m)v(0)+1+\G+(1-\G) P_a\,.
\end{equation*}
which can immediately solved, obtaining
\begin{equation}
\label{har333}
F(e^{\imath\th})=\frac{(1-\m)v(0)+1+\G}{\m-P_a(e^{\imath\th})}+
\frac{(1-\G)P_a(e^{\imath\th})}{\m-P_a(e^{\imath\th})}\,.
\end{equation}
Consider for $\l>\l_*$ (thus $P_a(e^{\imath\th})\equiv P_{a(\l)}(e^{\imath\th})$ and
$\m\equiv\m(\l)$) the following elements of $H^2(\bt)$ given by,
\begin{align*}
&G(e^{\imath\th}):=P_{H^2(\bt)}\left[\frac{1}{\m-P_a}\right](e^{\imath\th})
=\sum_{k\geq0}g_ke^{\imath k\th}\,,\\
&H(e^{\imath\th}):=P_{H^2(\bt)}\left[\frac{P_a}{\m-P_a}\right](e^{\imath\th})
=\sum_{k\geq0}h_ke^{\imath k\th}\,.
\end{align*}
It is well known that the above functions can be analytically
continued inside the unit circle simply by replacing
$e^{\imath\th}\to z$, see e.g. \cite{R}, Chapter 17. We call such
functions as $G(z)$ and $H(z)$, respectively. Thus,
\begin{align}
\label{har4}
&\frac1{2\pi}\int_0^{2\pi}G(e^{\imath\th})\di\th=g_0\,,\quad
\frac1{2\pi}\int_0^{2\pi}H(e^{\imath\th})\di\th=h_0\,,\nn\\
&G:=G(a)-\frac1{2\pi}\int_0^{2\pi}G(e^{\imath\th})\di\th=\sum_{k=1}^{+\infty}g_ka^{k}\,,\\
&H:=H(a)-\frac1{2\pi}\int_0^{2\pi}H(e^{\imath\th})\di\th=\sum_{k=1}^{+\infty}h_ka^{k}\,.\nn
\end{align}
Notice that, after denoting the analytic continuation of $f$ inside the unit circle as $f(z)$,
\begin{align}
\label{har5}
&P_{H^2(\bt)}F(e^{\imath\th})
=f(e^{\imath\th})\,,\nn\\
&\frac1{2\pi}\int_0^{2\pi}f(e^{\imath\th})\di\th=v(0)\,,\\
&f(a)-\frac1{2\pi}\int_0^{2\pi}f(e^{\imath\th})\di\th=\G\,.\nn
\end{align}
After integrating on the unit circle first, and then evaluating the projection
on $H^2(\bt)$ of both members of \eqref{har333} at $a\equiv a(\l)$,
we obtain by taking into account \eqref{har4} and \eqref{har5}, 
\begin{align*}
v(0)&=[(1-\m)v(0)+1+\G]g_0+(1-\G)h_0\,,\\
\G&=[(1-\m)v(0)+1+\G]G+(1-\G)H\,,
\end{align*}
respectively. This leads to the
following linear system for the unknown $v(0)$ and $\G$,
$$
\begin{cases}
\left[1+(\m-1)g_0\right]v(0)&+\,\,\,\quad (h_0-g_0)\G=g_0+h_0\,,\\
(\m-1)Gv(0)&+(1+H-G)\G=G+H\,,
\end{cases}
$$
which has a unique solution if $\l>\l_*$. For $v(0)$ this leads to,
\begin{equation}
\label{har6}
v(0)=\frac{g_0+h_0+2(g_0H-h_0G)}{1-G+H+(\m-1)(g_0+g_0H-h_0G)}\,.
\end{equation}
 The first step is to compute the analytic continuation of $\frac1{\m-P_a(e^{\imath\th})}$ and
$\frac{P_a(e^{\imath\th})}{\m-P_a(e^{\imath\th})}$ inside the annulus
$\{z\in\bc : z_-<|z|<z_+\}$, where $z_-$, $z_+$ are given in \eqref{zed}. This leads to
\begin{align*}
\frac1{\m-P_a(z)}=&\frac{(1+a^2)z-a(1+z^2)}{\sqrt{\D}}\S(z)\,,\\
\frac{P_a(z)}{\m-P_a(z)}=&\frac{(1-a^2)z}{\sqrt{\D}}\S(z)\,,
\end{align*}
where
$$
\S(z):=\frac1{z_+}\sum_{k=0}^{+\infty}\left(\frac{z}{z_+}\right)^k
+\frac1{z}\sum_{k=0}^{+\infty}\left(\frac{z_-}{z}\right)^k\,.
$$
We get for $g_0$, $h_0$, $G$, $H$ appearing in \eqref{har6},
\begin{align}
\label{gghh}
g_0=&\frac{a}{\sqrt{\D}}\left[\left(a+\frac1a\right)-\left(z_-+\frac1{z_+}\right)\right]\,,\nn\\
G=&\frac{a^2}{\sqrt{\D}(z_+-a)}\left[\left(a+\frac1a\right)-\left(z_++\frac1{z_+}\right)\right]\,,\\
h_0=&\frac{(1-a^2)}{\sqrt{\D}}\,,\quad H=\frac{a(1-a^2)}{\sqrt{\D}(z_+-a)}\nn\,.
\end{align}
The last step is to insert \eqref{gghh} in \eqref{har6} and compute the limit $\l\downarrow\l_*$. By taking into account that, first $\D\to0$ and correspondingly $z_\pm\to1$, and then
$\m\to\frac{1+a(\l_*)}{1-a(\l_*)}$, we obtain for the limit of $v(0)$ (which is a function of $\l$),
$$
\lim_{\l\downarrow\l_*}\big\langle R_{A_{\bh^{Q}}}(\l)\d_0,\d_0\big\rangle
=\lim_{\l\downarrow\l_*}v(0)=\frac{1-a(\l_*)}{a(\l_*)}
$$
which is finite, that is $A_{\bh^{Q}}$ is transient.
\end{proof}

\section{the perturbed tree of order $Q$ along a subtree of order $q$}

The present section is devoted to $\bg^{Q,q}$, $2<q\leq Q$,  obtained by adding to
$\bg^{Q}$ self loops on vertices of the subtree $S\sim\bg^{q}$, see Fig. \ref{Fig8}. The main object is the operator
$T_{a,q}$ on $\ell^2(\bg^{q})$ which is the convolution by the function $f_a(x):=a^{d(x,0)}$. Such a convolution operator is well defined if $a$ is sufficiently small, see below.
It extends the previous case when $q=2$ and then $T_{a,2}\equiv T_a$ treated in Section \ref{z}. We refer the reader to \cite{FP} for the detailed exposition of the basic harmonic analysis on the Cayley Trees, and for further details.
 \begin{figure}[ht]
     \centering
     \psfig{file=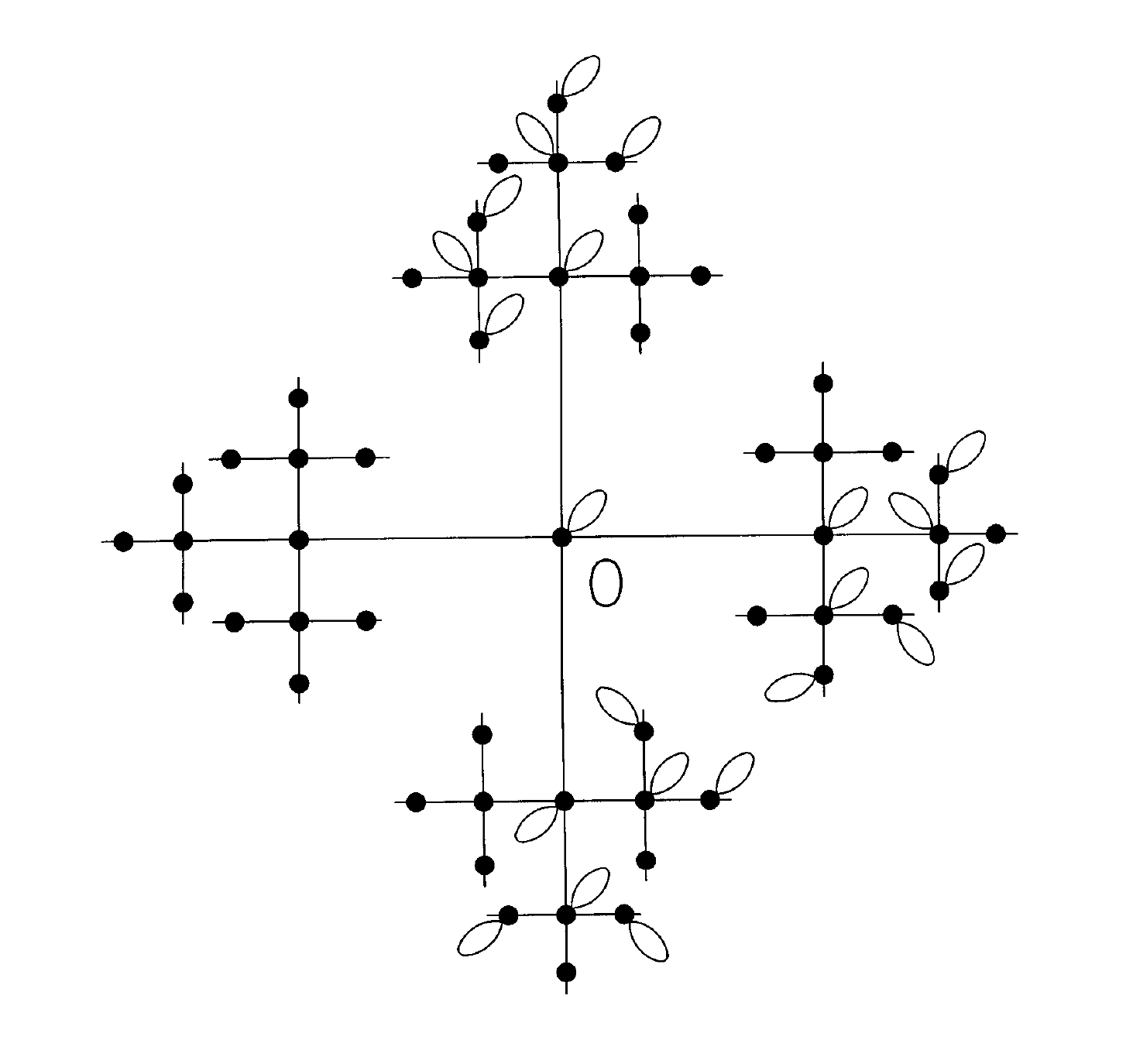,height=2.1in} \quad \psfig{file=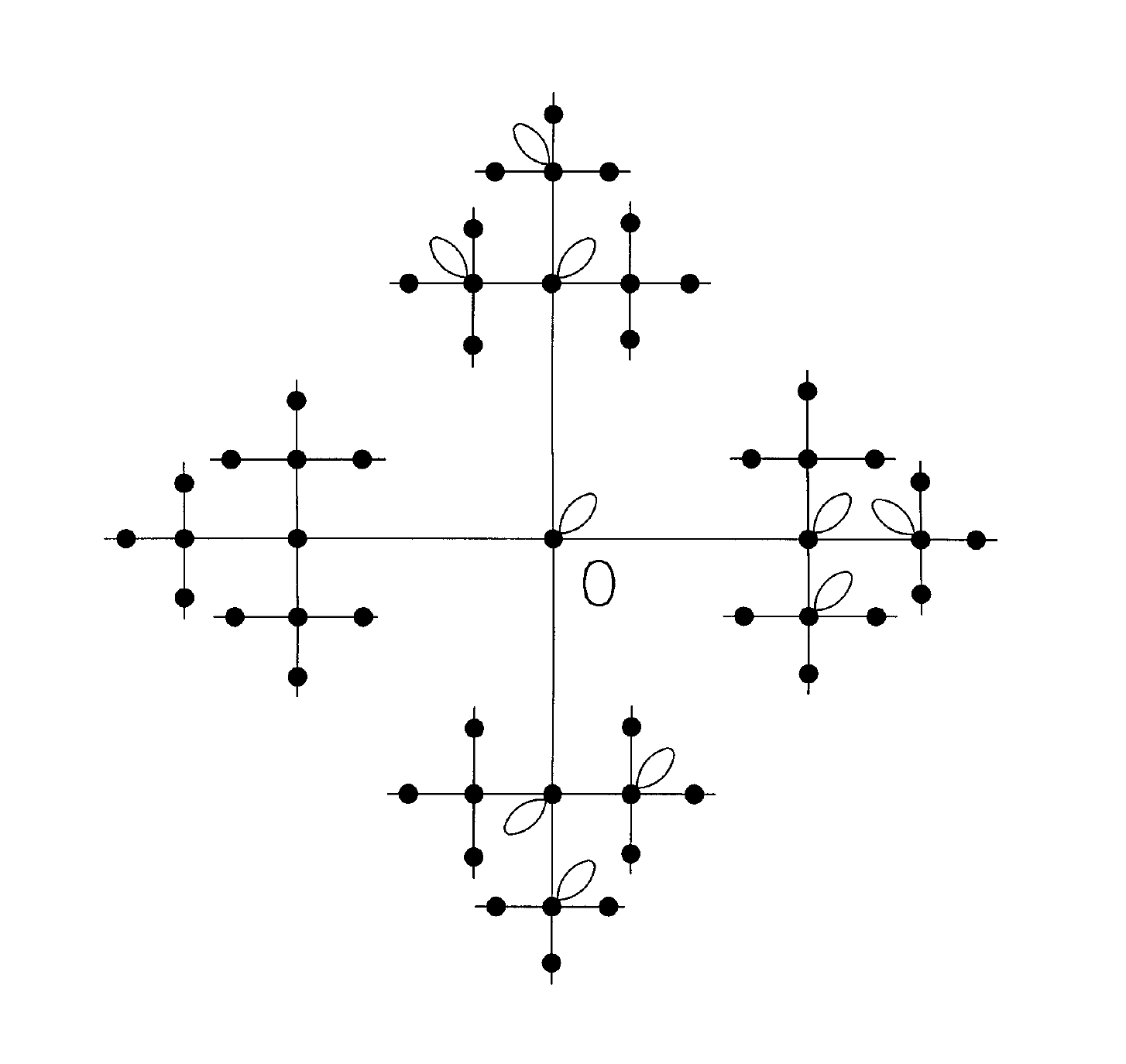,height=2.1in}
     \caption{The networks $\bg^{4,3}$ and $Z_2$.}
     \label{Fig8}
     \end{figure}

We start by considering the convolution by the functions $\m_n$ with
$$
\m_n:=|\G_n|^{-1}\chi_{\G_n}
$$
supported on the shell $\G_n$ made of the points at the distance $n$
from the root $0$. In our computations for objects involving
$\bg^{Q,q}$, the parameter $a$ will be a function of $Q$ and $\l$,
see below. As $Q$ will be kept fixed for all the computations, we do
not explicitly report such a dependence on $Q$ in the parameters
under consideration, like $a=a(\l,Q)$, $\m=\m(\l,Q)$ given in
\eqref{000} and \eqref{333}, when it causes no confusion.

For the convolution operator $T_{a,q}$, we get
$$
T_{a,q}=\idd_{\ell^2(\bg^{q})}+\frac{q}{q-1}\sum_{k=1}^{+\infty}[a(q-1)]^k\m_k\,,
$$
and by taking into account that $\m_k$ is a polynomial function $Q_k(\m_1)$ (cf. \cite{FP}, Section 3.1), we formally write
$$
T_{a,q}=\idd_{\ell^2(\bg^{q})}+\frac{q}{q-1}\sum_{k=1}^{+\infty}[a(q-1)]^kQ_k(\m_1)\,.
$$
 This means that $T_{a,q}=f(\m_1)$ by the analytic functional calculus of $\m_1$ where the function
\begin{equation}
\label{man1}
f(w)=1+\frac{q}{q-1}\sum_{k=1}^{+\infty}[a(q-1)]^kQ_k(w)
\end{equation}
is analytic at least in a neighborhood of the spectrum of $\m_1$, the last being the segment
$\left[-\frac{2\sqrt{q-1}}{q},\frac{2\sqrt{q-1}}{q}\right]$, see \cite{FP},
Theorem 3.3.3.\footnote{\label{foot}It can be seen that $f(w)$ is analytic on a neighborhood of the spectrum of $\m_1$, provided that
$a\sqrt{q-1}<1$.} It is standard to compute $F:=f\circ\g$, where
$\g$ is the function
$$
\g(z):=\frac{(q-1)^z+(q-1)^{1-z}}{q}
$$
given in pag. 40 of \cite{FP}. In addition, one can recover from the computations in \cite{FP}, that $Q_k(\g(z))=\f_z(k)$ where $\f_z$ is the spherical function appearing in Theorem 3.2.2 of \cite{FP}. By taking into account the previous considerations and after some standard computations, we obtain
\begin{equation}
\label{man}
F(z)=1+\frac{a}{(q-1)^{1-z}-(q-1)^{z}}\left[\frac{(q-1)^{2(1-z)}-1}{1-a(q-1)^{1-z}}
+\frac{1-(q-1)^{2z}}{1-a(q-1)^{z}}\right]\,.
\end{equation}
After removing the removable singularities for
$z=\imath k\pi/\ln(q-1)$, if $a$ is sufficiently small (cf. Footnote \ref{foot}) $F$ is analytic in a neighborhood of the line
$\left\{z\in\bc : \rel(z)=\frac12\right\}$, which is precisely the inverse image under $\g$ of
$\left[-\frac{2\sqrt{q-1}}{q},\frac{2\sqrt{q-1}}{q}\right]$. We then have the following
\begin{Prop}
\label{fc}
If $a<\frac1{\sqrt{q-1}}$, then $\|T_{a,q}\|=\frac{1-a^2}{(1-a\sqrt{q-1})^2}$.
\end{Prop}
\begin{proof}
If $a<\frac1{\sqrt{q-1}}$, $f$ is analytic in a neighborhood of the spectrum of $\m_1$ and $T_{a,q}=f(\m_1)$. Thanks to the Spectral Mapping Theorem (cf. \cite{T}, Proposition I.2.8) and the fact that $\m_1$ is selfadjoint,
$$
\|T_{a,q}\|=\spr(T_{a,q})
=\max_{w\in\left[-2\sqrt{q-1}/q,2\sqrt{q-1}/q\right]}\left|f(w)\right|
=\max_{z\in\left\{z\in\bc : \rel(z)=1/2\right\}}\left|F(z)\right|\,.
$$
Now,
$$
F(1/2+\imath\th/\ln(q-1))=\frac{1-a^2}{1-2a\sqrt{q-1}\cos\th+a^2(q-1)}
$$
which is maximum whenever $\th=0$ and the assertion follows.
\end{proof}
\begin{Prop}
\label{norm}
For each fixed $q$ there exists a $Q(q)>q$ such that $\|A_{\bg^{Q,q}}\|>\|A_{\bg^{Q}}\|$ provided that
$q\leq Q\leq Q(q)$. Such an upper bound is given by
\begin{equation}
\label{ub}
Q(q)=\bigg[\left(2\sqrt{q-1}+1+\sqrt{4\sqrt{q-1}+1}\right)^2\bigg/4\bigg]+1\,.
\end{equation}
\end{Prop}
\begin{proof}
We start by noticing that $a(\l)\equiv a(\l,Q)$ decreases as $\l$ increases. In addition
$a\left(\|A_{\bg^{Q}}\|\right)=\frac1{\sqrt{Q-1}}$. This means that $a(\l)\sqrt{q-1}<1$ for each
$\l\geq\|A_{\bg^{Q}}\|$. By taking into account Proposition \ref{fc} and Theorem \ref{es}, the secular equation \eqref{sc} for the adjacency of $\bg^{Q,q}$ becomes
\begin{equation}
\label{hss}
\frac{1-a^2}{(1-a\sqrt{q-1})^2}=\m
\end{equation}
where $a$ and $\m$, given by \eqref{000}, \eqref{333} respectively, are functions of $\l$ and $Q$.
Thanks to the fact that in \eqref{hss} the l.h.s. is decreasing and the r.h.s. is increasing, whenever
$\l\geq\|A_{\bg^{Q}}\|$ increases, in order to determine $Q(q)$ it is enough to solve \eqref{hss} w.r.t. $Q$ after putting $a=a\left(\|A_{\bg^{Q}}\|\right)$ and
$\m=\m\left(\|A_{\bg^{Q}}\|\right)=\frac{Q-2}{\sqrt{Q-1}}$. By defining $x:=\sqrt{Q-1}$, $b:=\sqrt{q-1}$,
\eqref{hss} becomes
$$
\frac{1-\frac1{x^2}}{\left(1-\frac{b}{x}\right)^2}=\frac{x^2-1}{x}\,,
$$
which has as the unique acceptable solution
$$
x=\frac{2b+1+\sqrt{4b+1}}2\,.
$$
\end{proof}
We have proven the following fact. Fix $q\geq2$, then there exists a unique $Q(q)$ given by
\eqref{ub} such that $q\leq Q\leq Q(q)$ implies $\|A_{\bg^Q}\|<\|A_{\bg^{Q,q}}\|=:\l_*$. As before, if
$Q>Q(q)$, the perturbation is too small to change the norm of the adjacency
operator, and then to create an hidden spectrum zone.

We pass on to the study of the Perron Frobenius eigenvector for
$A_{\bg^{Q,q}}$ when $2<q\leq Q$. As in the previous sections, we
consider the subgraph $\bg^q_n\subset\bg^{Q}$ made of the finite
volume subtree of order $q$ centered on the root
$0\in\bg^{q}\subset\bg^{Q}$. As the adjacency of the graph $Z_n$
(cf. Fig. \ref{Fig8}) obtained by perturbing $\bg^{Q}$ with
self loops along $\bg^q_n$ is recurrent, it has a unique
Perron Frobenius eigenvector $v_n$, normalized such that
$v_n(0)=1$, where $0$ is the common root for all the graphs under
consideration.

The main properties of the Perron Frobenius eigenvector are summarized in the following
\begin{Thm}
\label{pftr}
Suppose $q\leq Q\leq Q(q)$. With the above notations, $v_n$ converges pointwise to a weight
$v$ which is a Perron--Frobenius eigenvector for
$A_{\bg^{Q,q}}$. It is given by
$$
v(x)=a(\l_*,Q)^{d(x,\bg^{q})}\f_{1/2}(y(x))\,,
$$
where, $y(x)\in\bg^{q}$ is described in Lemma \ref{loo},
$a(\l_*,Q)$ is given by \eqref{000}, $\l_*$ is the unique solution of \eqref{hss}, and finally
$\f_{1/2}$ is the function on the tree $\bg^{q}$ given in Theorem 3.2.2 of \cite{FP}, by
$$
\f_{1/2}(x)=\left(1+\frac{q-2}{q}d(x,0)\right)(q-1)^{-\frac{d(x,0)}2}\,.
$$
\end{Thm}
\begin{proof}
We have previously shown that $\|T_{a,q}\|=F(1/2)=f(\frac{2\sqrt{q-1}}{q})$, where
$F$ and $f$ are given in \eqref{man} and \eqref{man1} respectively, and finally
$\frac{2\sqrt{q-1}}{q}=\|\m_1\|$. In addition,
$$
(\m_1*\f_{1/2})(x)=\|\m_1\|\f_{1/2}(x)\,.
$$
As
$$
\|\m_n\|=\max_{z\in\left\{z\in\bc : \rel(z)=1/2\right\}}\left|\f_z(n)\right|\,,
$$
we compute
$$
\f_{1/2+\imath\th/\ln(q-1)}=\frac{(q-2)I_n(\th)\cos\th+q\cos n\th}{q(q-1)^{n/2}}\,,
$$
where
$$
I_n(\th)=\frac{\sin n\th}{\sin\th}
$$
if $\th\neq k\pi$, and $\pm n$ according to the parity of $k$ and $nk$, when $\th=k\pi$. Now, the $I_n$ satisfy the recursive equation
$$
I_0(\th)=0,\quad I_{n+1}(\th)=I_n(\th)\cos\th+\cos n\th\,.
$$
This means that $|I_n(\th)|$ attains its maximum when $\th=2k\pi$, which implies
$\|\m_n\|=\f_{1/2}(n)$ and
$$
(\m_n*\f_{1/2})(x)=\|\m_n\|\f_{1/2}(x)\,.
$$
Now, thanks to the Monotone Convergence Theorem, we get
\begin{align*}
\left(T_{a,q}\f_{1/2}\right)(n)
&=\f_{1/2}(n)+\frac{q}{q-1}\left\{\left[\sum_{k=1}^{+\infty}(a(q-1))^k\m_k\right]*\f_{1/2}\right\}(n)\\
&=\f_{1/2}(n)+\frac{q}{q-1}\sum_{k=1}^{+\infty}(a(q-1))^k(\m_k*\f_{1/2})(n)\\
&=\f_{1/2}(n)+\frac{q}{q-1}\sum_{k=1}^{+\infty}(a(q-1))^k\f_{1/2}(k)\f_{1/2}(n)\\
&=\left[1+\frac{q}{q-1}\sum_{k=1}^{+\infty}(a(q-1))^k\f_{1/2}(k)\right]\f_{1/2}(n)\\
&=F(1/2)\f_{1/2}(n)=\|T_{a,q}\|\f_{1/2}(n)\,.
\end{align*}
Namely, $\f_{1/2}$ is a (generalized) Perron Frobenius eigenvector
for $T_{a,q}$ as well.\footnote{The fact that $\f_{1/2}(d(x,0))$ is
a Perron Frobenius weight for $T_{a,q}$ automatically follows from
the second part of the proof. We decided to give a different proof
as it does not depend on the approximation procedure by
finite volume Perron Frobenius eigenvectors.}

In order to show that $v$ is attained as the pointwise limit of the sequence of the finite volume Perron Frobenius eigenvectors $v_n$ of the graphs $Z_n$, it is enough to show that $\f_{1/2}$ is the
pointwise limit of the Frobenius eigenvectors $w_n$ for
$P_{\ell^2(\bg^q_n)}T_{a(\l_n),q}P_{\ell^2(\bg^q_n)}$, normalized to $1$ at the root $0$
(and eventually extended at zero outside the ball of radius $n$). As usual $a(\l)\equiv a(\l,Q)$, and $\l_n=\|A_{Z_n}\|$.

By symmetry, all the $w_n$ are radial functions. Thus, after summing
up the "angular part", we reduces the matter to a situation similar
to that in Theorem \ref{czz} involving a positivity preserving
operator acting on the Hilbert space $L^2(\bn,\psi\di\n)$ made of
the $\ell^2$--radial functions on $\bg^q$, where $\n$ is the
counting measure, and the density $\psi(n)=|\G_n|$. Namely, we
suppose
$\La_n:=\|P_{\ell^2(\bg^q_n)}T_{a(\l_n),q}P_{\ell^2(\bg^q_n)}\|$
fixed throughout the computation at the step $n$. Define for
$k=0,1,\dots,n$, $n\in\bn$,
\begin{equation}
\label{dsigm}
\s_n(k):=(q-1)^ka(\l_n)^kw_n(k)\,.
\end{equation}
As before  (cf. Lemma \ref{00}),
$$
\La_n\uparrow\La_*:=\|T_{a(\l_*),q}\|=\frac{1-a(\l_*)^2}{(1-a(\l_*)\sqrt{q-1})^2}\,,
$$
thanks to the fact that $\m(\l_n)\uparrow\m(\l_*)$.
In addition, we have also $a(\l_n)\downarrow a(\l_*)$, where $\l_*$ is the unique solution of the secular equation \eqref{sc} for the situation under consideration.
Put $\S_N:=\frac{(q-1)\La_N+1}q$,
\begin{align*}
&\d_0(a):=1\,,\quad \d_1(a):=1+(q-1)a^2\,,\\
&\d_n(a):=1+(q-2)\sum_{l=1}^{n-1}(q-1)^{l-1}a^{2l}+(q-1)^na^{2n}\,,n>0\,.
\end{align*}
By taking into account \eqref{dsigm},
after some tedious computations we can see that the solution for the $\s_N(n)$,
$0\leq n\leq N$, $N\in\bn$ is given by
\begin{align}
\label{iiffa}
\s_N(n)=&\frac1{\La_N}\bigg\{\sum_{m=0}^{n-1}
\left[(q-1)^{n-m}a(\l_N)^{2(n-m)}\d_m(a(\l_N))-\d_n(a(\l_N))\right]\s_N(m)\nn\\
+&\d_n(a(\l_N))\S_N\bigg\}\,.
\end{align}
Namely, the form of the system defining the $\s_n$ in terms 
of $\La_n$ is triangular and independent on the size (i.e. on
$n\in\bn$). It follows from the previous claims , thanks to the fact
that $a(\l_n)\to a(\l_*)$ and $\La_n\to\|T_{a(\l_*),q}\|$, that
$\s_n(k)$ converges pointwise in $k$ when $n\to\infty$. The proof
will be complete if we show that \eqref{iiffa} is satisfied for the
sequence $\{\s(n)\}$, with $\La=\|T_{a(\l_*),q}\|$ and
$\s(n)=(q-1)^na(\l_*)^n\f_{1/2}(n)$. To this end, after denoting as
usual $a=a(\l_*)$, we apply the inductive hypothesis and
\eqref{iiffa} becomes
\begin{equation}
\label{iiffa1}
\La(\s(n+1)-\xi^2\s(n))=(1-a^2)(\S-R_n)\,,
\end{equation}
where $\xi:=a\sqrt{q-1}$, and
$R_n:=\sum_{k=0}^n\s(k)$.
By inserting in \eqref{iiffa1},
\begin{align*}
&R_n=\frac{1-\xi^{n+1}}{1-\xi}+\xi\frac{(q-2)[1-\xi^{n+1}-(n+1)\xi^n(1-\xi)]}{q(1-\xi)^2}\,,\\
&\La=\frac{1-a^2}{(1-\xi)^2}\,,\qquad\S=\frac{(q-1)(1-a^2)}{q(1-\xi)^2}+\frac1q\,,
\end{align*}
we get that it becomes an identity and the proof follows.
\end{proof}
Notice that the above proof works even in the case when $q=2$. Namely, we get an alternative proof of the fact that the finite volume Perron Frobenius eigenvectors of $A_{\bg^{Q,2}}$ converge pointwise to \eqref{pfz} which is the unique Perron Frobenius generalized eigenvector as
$A_{\bg^{Q,2}}$ is recurrent.

We now move on to study the resolvent of $T_{a,q}$ for $q\leq Q\leq
Q(q)$ and $\l>\|A_{\bg^{Q,q}}\|$. It has the form
$$
R_{T_{a,q}}(\m)=\frac1{2\pi\imath}\oint\frac{R_{\m_1}(w)}{\m-f(w)}\di w\,,
$$
where $f$ is the function given in \eqref{man1}, and the integral is
over a small ellipse, oriented counterclockwise around the spectrum
of $\m_1$. By doing a standard change of variable, we get
$$
R_{T_{a,q}}(\m)=\frac1{2\pi\imath}\int_{\ell_{\eps}}\frac{R_{\m_1}(\g(z))}{\m-F(z)}\g'(z)\di z\,,
$$
where $F$ is given in \eqref{man} and
$\ell_{\eps}=\left\{z\in\bc : \rel(z)=1/2+\eps\,,0\leq\im(z)\leq2\pi/\ln(q-1)\right\}$ for all the sufficiently small $\eps>0$.
By taking into account the computation of $R_{\m_1}(\g(z))$ given in Theorem 3.3.3 of \cite{FP} and the derivative $\g'(z)$, we get
\begin{equation}
\label{man2}
\big\langle R_{A_{\bg^{Q,q}}}(\l)\d_0,\d_0\big\rangle
=\frac{\ln(q-1)}{2\pi\imath}\int_{\ell_{\eps}}\frac{\big[(q-1)^z-(q-1)^{1-z}\big] F(z)}
{\big[(q-1)^z-(q-1)^{-z}\big](\m-F(z))}\di z\,,
\end{equation}
where $a$ (appearing in the definition of $F(z)$) and $\m$ depend on
$\l$ and $Q$. Now, in order to have a more manageable formula, we
introduce a new variable by putting $\z:=(q-1)^z$ in \eqref{man2}.
This leads to
\begin{Lemma}
If $q\leq Q\leq Q(q)$ and $\l>\|A_{\bg^{Q,q}}\|$, we have for the
the following representation,
\begin{align}
\label{man3}
\big\langle &R_{A_{\bg^{Q,q}}}(\l)\d_0,\d_0\big\rangle\\
=\frac{a^2-1}{2\pi\imath}&\oint_{C_{\sqrt{(q-1)}}}\frac{[z^2-(q-1)]\di z}
{(z^2-1)\{a\m z^2-[(1+a^2(q-1))\m-(1-a^2)]z+a(q-1)\m\}}\nn\,,
\end{align}
where $a$ and $\m$ are function of $\l$ and $Q$, the integral is on
the circle $C_{\sqrt{(q-1)}}$  of radius $\sqrt{q-1}$, centered at
the origin and  oriented counterclockwise.
\end{Lemma}
\begin{proof}
After the change of the variable previously explained, the integrand in \eqref{man2} becomes proportional to that in \eqref{man3}, and $\ell_{\eps}$ becomes a circle $C_{\sqrt{q-1}+\d}$
centered in the origin whose radius is $\sqrt{q-1}+\d$ for any sufficiently small $0<\d< d$, with  $d>0$ depending on $\l$, for $\l>\l_*$ close to $\l_*$ (or equivalently as $\m$ is close to
$\frac{1-a^2}{(1-a\sqrt{q-1})^2}$).
As explained in Section \ref{z}, it is straightforward to show that
$\l>\l_*$ corresponds to $\m>\frac{1-a^2}{(1-a\sqrt{q-1})^2}$. In addition, $\l>\l_*$ implies
$$
\D:=[(1+a^2(q-1))\m-(1-a^2)]^2-4\m^2a^2(q-1)>0\,.
$$
The last is zero if $\l=\l_*$, or equivalently if
$\m=\frac{1-a^2}{(1-a\sqrt{q-1})^2}$. This means that the four
simple poles of the integrand in \eqref{man3} are precisely $\pm1$
and $z_{\pm}$ with $z_-<\sqrt{q-1}$ and $\sqrt{q-1}+d<z_+$, for some
$d>0$. Thus, we can replace in \eqref{man3}, the circle
$C_{\sqrt{q-1}+\d}$ directly with $C_{\sqrt{q-1}}$.
\end{proof}
We are ready to establish the main properties of the resolvent of $A_{\bg^{Q,q}}$ which are summarized in the following
\begin{Thm}
Suppose that $q\leq Q\leq Q(q)$. If $\l>\|A_{\bg^{Q,q}}\|$, we have
\begin{equation}
\label{resss3}
R_{A_{\bg^{Q,q}}}(\l)=R_{A_{\bg^{Q}}}(\l)\bigg(\idd_{\ell^2(\bg^{q})} +P_{\ell^2(\bg^{q})}
\bigg(\idd_{\ell^2(\bg^{Q})} -\frac1{\l}W\bigg(\frac1{\l}\bigg)\bigg)^{-1}P_{\ell^2(\bg^{q})}
R_{A_{\bg^{Q}}}(\l)\bigg)\,,
\end{equation}
where $W$ is the operator acting on $\bg^{q}$ given by \eqref{ww}. In addition, $\bg^{Q,q}$ is transient.
\end{Thm}
\begin{proof}
As explained in the analogous previous results, we have to only
prove the transience by \eqref{resss3}. This leads to
\eqref{man3}, or equivalently
$$
\big\langle R_{A_{\bg^{Q,q}}}(\l)\d_0,\d_0\big\rangle
=\frac{a^2-1}{a\m}\frac1{2\pi\imath}\oint_{C_{\sqrt{(q-1)}}}\frac{[z^2-(q-1)]\di z}
{(z-1)(z+1)(z-z_-)(z-z_+)}\,.
$$
This can be computed by the Residue Theorem and, by taking into account that
$z_+\downarrow\sqrt{q-1}$, $z_-\uparrow\sqrt{q-1}$ as $\l\downarrow\|A_{\bg^{Q,q}}\|$ we conclude that the unique term which might be divergent is that containing
$$
\frac{\sqrt{q-1}-z_-}{z_+-z_-}=\frac12\left[1-\frac{[(1+a^2(q-1))\m-(1-a^2)]-2a\m\sqrt{q-1}}
{\sqrt{[(1+a^2(q-1))\m-(1-a^2)]^2-4a^2\m^2(q-1)}}\right]\,.
$$
We obtain, by taking into account that $a=a(\l,Q)$, $\m=\m(\l,Q)$,
\begin{align*}
&\lim_{\l\downarrow\l_*}\big\langle R_{A_{\bg^{Q,q}}}(\l)\d_0,\d_0\big\rangle
=\frac{(1-a(\l_*,Q)\sqrt{q-1})^2\sqrt{q-1}}{a(\l_*,Q)(q-2)}\\
\times&\left\{1+
\lim_{\l\downarrow\l_*}
\sqrt{\frac{[(1+a(\l,Q)^2(q-1))\m(\l,Q)-(1-a(\l,Q)^2)]-2a(\l,Q)\m(\l,Q)\sqrt{q-1}}
{[(1+a(\l,Q)^2(q-1))\m(\l,Q)-(1-a(\l,Q)^2)]+2a(\l,Q)\m(\l,Q)\sqrt{q-1}}}\right\}\\
=&\frac{(1-a(\l_*,Q)\sqrt{q-1})^2\sqrt{q-1}}{a(\l_*,Q)(q-2)}
\end{align*}
as $\m(\l,Q)\to\frac{1-a(\l_*,Q)^2}{(1-a(\l_*,Q)\sqrt{q-1})^2}$ when $\l\to\l_*$.
\end{proof}

\medskip\par\noindent{\it Acknowledgements.}
The author would like to thank L. Accardi for the invitation to the
7th Volterra--CIRM International School: "Quantum Probability and
Spectral Analysis on Large Graphs", which  was extremely inspiring
for the mathematical investigation of the Bose Einstein condensation
on non homogeneous networks. He is grateful to M. Picardello for
several helpful discussions, and for introducing and explaining most
of the results of the Section 3 of the book \cite{FP} which was
crucial for the present work. He is also grateful to M.
Cirillo for some explanation concerning the physical and
experimental meaning of the pure hopping model. Finally, he
acknowledges the warm hospitality extended by the International
Islamic University Malaysia during the period January--April 2009,
when the present work started.


\end{document}